\documentclass[11pt,leqno]{amsart}
 \usepackage[margin=1in]{geometry}
\usepackage{xr-hyper}

\usepackage{overpic,wrapfig}

\usepackage{amssymb, amsmath, graphicx, comment, enumerate}
\usepackage[dvipsnames]{xcolor}
\usepackage{amscd}
\usepackage{algorithm}
\usepackage{algorithmic}
\usepackage{tikz-cd}
\usepackage{url}

\numberwithin{equation}{section}

\newtheorem{clai}{Claim}[section]

\newtheorem{cor}[clai]{Corollary}
\newtheorem{defn}[clai]{Definition}
\newtheorem{exam}[clai]{Example}

\newtheorem{lemm}[clai]{Lemma}

\newtheorem{obse}[clai]{Observation}

\newtheorem{prop}[clai]{Proposition}

\newtheorem{rema}[clai]{Remark}

\newtheorem{theo}[clai]{Theorem}
\numberwithin{equation}{section}

\newcommand{\bbbZ}{{\mathbb Z}}

\def \mb {\mathbb}

\newcommand{\Aut}{{\rm Aut}}

\newcommand{\In}{{\rm In}}

\newcommand{\nbd}{{\rm Nbd}}
\newcommand{\Out}{{\rm Out}}

\def \rar {\rightarrow}

\newcommand{\ttimes}[1]{{\,{\times}_{#1}\,}}

\newcommand{\comm}[1]{}

\begin{document}

\title{Lifting Voltages in Graph Covers}
\author{Nata\v sa Jonoska, Mil\'e Kraj\v cevski, \& Gregory L. McColm}
\address{
         Department of Mathematics \& Statistics, \\
         University of South Florida \\
         Tampa, FL 33620}
\email{ {jonoska@usf.edu}, {mile@usf.edu}, { mccolm@usf.edu}}

\begin{abstract}
We consider voltage digraphs, here referred to as graphs, whose edges are 
 labeled with elements from a given group, and explore their derived graphs. Given two voltage graphs, with voltages in abelian groups, 
we establish a necessary and sufficient condition for 
 their two derived graphs to be isomorphic.
This condition requires: (1) the existence of a voltage 
 graph that covers both given graphs, and (2)  
 when the two sets of voltages are lifted to the common cover, the correspondence between these sets of voltages determines 
 an isomorphism between the groups generated by these voltages. 
We show that conditions (1) and (2) are decidable, and provide  a method for constructing the common cover and for lifting the voltage assignments.

\end{abstract}

\maketitle

\section{Introduction}\label{S_Introduction}

Although voltage graphs - graphs and digraphs whose edges are labeled with elements from a given group~\cite{G74,GT} - were introduced to address certain problems in  topological 
 graph theory, they can also provide 
 tools to study 
 `large', or infinite graphs of 
  `high symmetry'.
For example, periodic graphs are often used to represent 
 covalent crystals, with the vertices representing atoms or 
 molecular building blocks, and the edges representing bonds 
 or ligands, respectively \cite{DFOK}.
 These graphs are characterized by 
 their automorphism groups as groups that  have a free 
 abelian subgroup of finite rank and finite index.  In \cite{Chung}, 
 periodic graphs were represented by finite 
 graphs, i.e., the quotient graphs,  
 (corresponding to  the 
 repeating unit) with voltages from a free abelian subgroup. Therefore a periodic digraph can be seen as a cover over a finite digraph associated with the unit of the periodic graph. The edges of the finite covered digraph can be labeled with group elements from a free abelian group acting on the periodic graph, forming a voltage graph
 whose derived graph is the original periodic digraph.
 
 One of the main questions that arises in the area of crystallographic structures is how to determine when two derived periodic graphs are isomorphic, in particular, when we are given only the quotient graphs of the two periodic graphs and the respective group elements assigned on their edges. 
 Periodic graphs arising from graphs  labeled with symbols from an alphabet were considered in~\cite{JKM,JKM2} through languages associated with labels of cyclic walks 
 that start and end in a given unit cell. These languages were recognized by deterministic counter machines, where the vector displacements of the transitions can be seen as voltages associated with the graphs of the counter machines. 
 
 Here we address the question of when derived graphs of two given finite connected voltage graphs are isomorphic. 
A more restricted question about the isomorphism of two derived graphs that are derived from a single graph but with two distinct sets of voltage assignments was studied in \cite{Skoviera} and \cite{KL0, KL1}.
Here we address a more general question when two non-isomorphic finite graphs with voltage assignments arise to isomorphic derived graphs. 
In answering this question we prove that whenever two derived graphs are isomorphic, there must be a common cover graph, covering both given voltage graphs.  After  constructing  the common cover of the two voltage graphs (that exists in certain conditions)  the voltages must be lifted.
When two sets of voltages are lifted to the common graph, the situation becomes the same to the one addressed by \cite{Skoviera} and \cite{KL0, KL1}. 
We also show that this common cover has a somewhat categorical property;  it is a minimal common cover such that every other common cover of the two given voltage graphs must factor through it (Lifting voltages has been addressed from categorical perspective in~\cite{Jenca2023voltage}). The success of the common cover construction depends on the groups that provide the voltages. 
We show that this method works if the voltage groups are 
 Dedekind - which includes abelian groups. 
In the case of abelian groups, we show that existence of an isomorphism of the (possibly infinite) derived graphs is decidable. 
However, we provide an example of a  pair of voltage graphs, when voltages are taken from groups  with subgroups that are not normal,   this method does 
 not provide an answer whether corresponding derived graphs are isomorphic.

We introduce preliminaries and notations in  
Section~\ref{preliminaries}. 
 Section~\ref{NoLabels}  reviews graph covers  and lifts.  It also shows that for two finite connected graphs with regular covers have a common cover. This common cover can be chosen to be minimal (up to an isomorphism) such that every other common cover factors through it. We also observe that this cover is a subgraph of the direct product of the two graphs. 
In Section~\ref{Svoltages}, we review voltage graphs and derived graphs. We introduce condensation of voltages, that is, reassigning voltages to the graph by fixing a spanning tree whose edges have voltage assignment the group identity. The condensation helps  to   lift  voltage assignments  so that the new derived graph is isomorphic to the original derived graph. 
Section~\ref{LVI} contains our main result. It shows a method to determine whether two finite connected voltage graphs have isomorphic 
   derived graphs. The method works in the case of voltages taken from a Dedekind group, and it  is constructive in case  of abelian groups. 
We conclude with a brief description of pairs of voltage graphs for 
 which this method is ineffective.

\section{Preliminaries}\label{preliminaries}

In this section, we review some notions of graph and group theory.
See \cite{Behzad,Rotman,GGT,DD} for details. 
We use the standard notation $\bbbZ$ 
 for integers. 

A {\em graph} is a quadruple $\Gamma = (V, E, \iota, 
 \tau)$, where $V$ is a set of {\em vertices}, $E$ is a 
 set of {\em edges}, and $\iota, \tau : E \to V$ 
 are {\em initial endpoint} and {\em terminal endpoint} functions, 
 respectively. In the combinatorics literature, such a structure is called 
  a {\em digraph} or a {\em multi-digraph}. However, in the rest of the exposition we allow walking in either direction on each edge, and the directions are used only to distinguish  associated group elements labeling the edges from their inverses (see Section~\ref{Svoltages}).
For edge sets and vertex sets, we do not use the subscript $\Gamma$ whenever the graph is understood from the context.  We  use notation $\iota$ and $\tau$ without subscripts as their domains are understood from the context. We also implicitly assume they have been defined for any given graph and take that a graph $\Gamma$ is given with the set of vertices and directed edges $(V,E)$. 
A {\em pointed graph} is a pair $(\Gamma, v)$ where $\Gamma$ is a graph and $v\in V$ is a 
 vertex of $\Gamma$, which we call the {\em base vertex}.

Given a graph $\Gamma = (V, E)$  
and a subset $V' \subseteq V$, the {\em 
 subgraph of $\Gamma$ induced by $V'$} is the graph $\Gamma' = (V', E')$  
 where $E' = \{ e\in E \,:\, \iota(e), \tau(e) \in V' \}$. 
%
The {\em order} of $\Gamma$ is the number of vertices, 
 denoted $|\Gamma|=|V|$,  while the {\em size} of $\Gamma$ is the number of edges, denoted $|E|$.

The {\em orientation} of edge $e$ is the ordered pair $(\iota(e), \tau(e))$, while 
 $\iota(e)$ and $\tau(e)$ are its {\em initial} and {\em terminal} {\em endpoints}. We also say that $e$ {\it starts} in $\iota(e)$
 and {\it terminates} at $\tau(e)$.
 
A vertex and an edge are {\em incident} if the vertex is an endpoint of the edge.
Two distinct vertices are {\em adjacent} if they are the endpoints of the same edge, while 
 two edges are {\em adjacent} it they share an endpoint.

A {\em loop} is an edge $e$ such that $\iota(e) = \tau(e)$.
A  pair of distinct edges $e$ and $e'$ such that $\iota(e) = \iota(e')$ 
 and $\tau(e) = \tau(e')$ are called {\em parallel}.

Given a graph $\Gamma$ and a vertex $v \in V$, the {\em inset} of $v$ is
 \[
 \In(v) = \{ e \in E : \tau(e) = v \}
 \]
 and its {\em indegree} is $|\In(v)|$, while the {\em outset} of $v$ is
 \[
 \Out(v) = \{ e \in E : \iota(e) = v \}
  \]
 and its {\em outdegree} is $|\Out(v)|$.
The {\em vertex figure} of a vertex $v$ is the set $\In(v)\cup \Out(v)$ of edges incident to $v$ 
 while the {\em neighborhood} of a vertex $v$ is the set $\nbd(v)\subseteq V$ of vertices 
 adjacent to $v$.

Given a graph $\Gamma$ with vertices $v, v' \in V$, a {\em walk} from $v$ to $v'$ is a 
 finite sequence of vertices and edges $v = v_0, e_1, v_1, e_2, v_2, e_3,\ldots, e_{n}, 
 v_n = v'$, $n\geq 0$, such that for each $i$, $1 \leq i \leq n$, $\{ v_{i-1}, v_{i}\} 
 = \{ \iota(e_{i}), \tau(e_{i}) \}$.
If $v_{i-1} = \iota(e_{i})$ and $v_{i} = \tau(e_{i})$, the walk traverses the edge $e_{i}$ 
 {\em along} $e_{i}$'s orientation, while if $v_{i-1} = \tau(e_{i})$ and $v_{i} = 
 \iota(e_{i})$, the walk traverses the edge $e_{i}$ {\em against} $e_{i}$'s orientation.
The {\em length} of the above walk is $n$, the number of edges traversed (counting 
 repetitions).

The {\em distance} between any two vertices in $\Gamma$ is the minimum length of any 
 walk from one of the vertices to the other.
A graph is {\em connected}, if for any two vertices in the graph, there is a walk from one 
 to the other.

For a walk $w$,  $\iota(w)$ is the first (or {\em initial}) vertex of the walk and 
 $\tau(w)$ is the last (or {\em terminal}) vertex of the walk;  if $\iota(w) = 
 \tau(w)$, then $w$ is {\em cyclic}. A graph is {\it acyclic} if the graph has no cyclic walks.
A walk in which no vertex occurs more than once is a {\em path} (except,  in a cyclic 
 path, the initial and the terminal vertex is the same).

Given two graphs $\Gamma_1 = (V_1, E_1)$ and $\Gamma_2 = (V_2, E_2)$, a {\em homomorphism} $\varphi : \Gamma_1 \to \Gamma_2$ consists 
 of a pair of functions $(\varphi_V ,\varphi_E )$ where $\varphi_V : V_1 \to V_2$ 
 and $\varphi_E : E_1 \to E_2$ such that for any $e\in E_1$,  $\iota(\varphi_E(e)) 
 = \varphi_V(\iota(e))$, and $\tau(\varphi_E(e)) = \varphi_V(\tau(e))$.
We  denote such a homomorphism simply by $\varphi : \Gamma_1 \to 
 \Gamma_2$ and use $\varphi$ without corresponding subscript when referring to 
 $\varphi_V$ or $\varphi_E$.
A bijective homomorphism's inverse (that is when both $\varphi_V$ and $\varphi_E$ 
 are bijections) is a homomorphism, and therefore is an {\em isomorphism}.
An isomorphism from a graph $\Gamma$ to itself is called an {\em automorphism}, and the group of all {\em automorphisms} of $\Gamma$ is denoted by $\Aut(\Gamma)$. For $g\in \Aut(\Gamma)$ acting on a vertex of an edge $x$ of $\Gamma$ we write $g\cdot x=gx$.

 For a given graph $\Gamma$ and a group $G\leq \Aut(\Gamma)$, we say that $G$ {\em acts 
 freely} on $\Gamma$ if, for each vertex $v \in V$, and each $g\not = 1_G\in G$, $gv \not = v$,
 where $1_G$ is the identity in $G$.
As automorphisms preserve the  orientations of edges, if $G$ acts freely on 
 $\Gamma$, then 
 for edge $e \in E$, $ge = e$ if and only if $g=1_G$.
The orbit of a vertex $v \in V$ is denoted with $G(v)$, i.e., $G(v)= \{ gv : g\in G \}$.
Similarly,  $G(e)$ denotes the orbit of an edge $e \in E$. 
The  orbits induce a partition on $V$, since $gu=hv$ implies that $u=g^{-1}hv$, and similarly for the edges.

Given a group $G$ and a subset $X \subseteq G$, $\langle G \rangle$ is the 
 subgroup of $G$ generated by $X$.

\section{Covers}\label{NoLabels}

The {\it quotient graph} $\Gamma / G=( V_{\Gamma/ G},  E_{\Gamma /G} )$
 is the graph whose vertices and edges are orbits of $\Gamma$, where $\iota (G(e)) = 
 G(\iota(e))$ and $\tau(G(e))=G(\tau(e))$.
The {\em canonical projection} $\phi : \Gamma \to \Gamma /G$, given by $\phi(v)  =G(v)$ 
 for $v \in V$ and $\phi(e) = G(e)$ for $e \in E$, is a graph homomorphism.
Following \cite{Sato} and recalling  that an {\em epimorphism} is an onto homomorphism, 
 we have:

\begin{defn}\label{cover_covering}
{\rm Let $\Gamma$ and $\Delta$ be graphs.
An epimorphism $\phi : \Gamma \to \Delta$ is  {\em covering}, 
 and $\Gamma$ is  {\em a cover} of $\Delta$ if, for each $v \in V_{\Gamma}$, the 
 restrictions $\phi{|}_{\Out(v)}$ and $\phi{|}_{\In(v)}$ are bijections.
Given pointed graphs $(\Gamma, v)$ and $(\Delta, u)$, a covering $\phi : \Gamma 
 \to \Delta$ is a {\em pointed covering} from $(\Gamma, v)$ to $(\Delta, u)$ if 
 $\phi(v) = u$.
 }
\end{defn}

The group $G_{\phi} = \{ g \in \Aut(\Gamma) : \phi g = \phi \}$ is the {\em cover group} 
 of the covering $\phi$.
Observe that the identity is always in $G_{\phi}$, hence $G_{\phi}$ is not empty.
The covering is {\em regular} if, for each vertex or edge $x$ of $\Delta$, $G_{\phi}$ acts 
 transitively 
 on $\phi^{-1}(x)$, that is, for each $y,z\in \phi^{-1}(x)$ there is $g\in G_\phi$ such that $gy=z$.
The following is from \cite[\S 5.2]{Sunada}, which we prove for the 
 convenience of the reader.

\begin{lemm}\label{lem:regular}
    If $\phi : \Gamma \to \Delta$ is regular, then $\Delta \cong \Gamma / G_{\phi}$ 
 via the canonical isomorphism $\phi(x) \mapsto G_{\phi}(x)$.
\end{lemm}

\begin{proof}
   The canonical map $\phi(x)\mapsto G_\phi(x)$ is well defined, because $\phi(x)=\phi(y)$ implies that for all $g\in G_\phi$
$\phi(gx)=\phi(x) =\phi(y)=\phi(gy)$, i.e, $x$ and $y$ have the same orbit, $G_\phi(x)=G_\phi(y)$. Observe that the injectivity holds because $\phi$ is regular, i.e., $G_\phi$ acts transitively on the preimage set $\phi^{-1}(x)$ for every edge, or vertex $x$ of $\Delta$, i.e., $\phi^{-1}(x)$ is precisely an orbit of $x$ with respect to $G_\phi$. Surjectivity follows because $\phi $ is covering.
\end{proof}

For the converse, notice that if $G \subseteq \Aut(\Gamma)$, then 
 the canonical covering $\Gamma \to \Gamma/G$ is regular.

The following bit of lore is an if and only if version of 
 \cite[Theorem 5.2]{Sunada}, and we include a proof for the convenience 
 of the reader.

\begin{prop}\label{FreeCover}
Let $\Gamma$ be a connected graph such that each vertex of $\Gamma$ has finite 
 indegree and outdegree, and let $G \leq \Aut(\Gamma)$.
Then $G$ acts freely on $\Gamma$ if and only if $\Gamma$ is a cover of $\Gamma / G$.
\end{prop}

\begin{proof}
Suppose $G$ acts freely on $\Gamma$ and let $\phi : \Gamma \to \Gamma / G$ be the 
 canonical projection.
Then for any vertex $v \in V_{\Gamma}$, if $e, e' \in \Out(v)$ and 
  $\phi(e) = \phi(e')$, then $G(e) = G(e')$ and there exists $g\in G$ with $ge = e'$.
But then $gv = g\iota(e) = \iota(ge) = \iota(e') = v$. As $G$ acts freely, 
 $g = 1_G$ and $e=e'$.
Thus all outgoing edges of $v$ must be in distinct orbits, so $\phi$ is one-to-one 
 on $\Out(v)$; similarly $\phi$ is one-to-one on $\In(v)$.

Conversely, suppose that $\Gamma$ is a cover of $\Gamma / G$, and let $gv = v$ for some 
 $g \in G$ and $v \in V$.
As $e, e' \in \Out(v)$ implies that $\phi(e) = \phi(e')$ if and 
 only if $e = e'$, the outgoing edges of $v$ must be in distinct 
 orbits; since $gv = v$, for each edge $e$ outgoing from to $v$,  
 $ge = e$.
But then for each $e \in \Out(v)$, $g\tau(e) = \tau(e)$ and for each $e \in \In(v)$, 
 $g\iota(e) = \iota(e)$.
Thus $gv' = v'$ for all 
$v' \in\nbd(v)$.
Proceeding by induction on the distance from $v$, one obtains $gv = v$ for all $v \in V$, 
 so $g$ must be the identity.
 \end{proof}

Consider the following example, from \cite{eon2}.

\begin{exam}\label{noncryst} {\rm
Consider the graph $\Gamma$ in Figure~\ref{nonbary}.
There are two orbits of vertices indicated as 
circles and squares. 
Call a pair of circular vertices {\em good} if they are ``in the same unit'' of four square 
 vertices, i.e., their incoming edges have the same initial vertex and their outgoing edges 
 share the same terminal vertex. 
For any set $S$ of pairs of good  
vertices, there is an automorphism of $\Gamma$ that fixes 
 all the other vertices and swaps the  
 vertices within a good pair in $S$;  thus the 
 automorphism group is uncountable and does not act freely on the graph in 
 Figure~\ref{nonbary}.
Let $G \subseteq \Aut(\Gamma)$ and suppose that $\Gamma / G$ is the graph at right in 
 Figure~\ref{nonbary}. 
Then for any pair of round vertices in the same unit, there is an element in $g \in G$ that 
 swaps them; notice that any such element must fix the two square vertices incident 
 to the two round vertices in that unit; it may fix or switch the other two vertices. Hence $G$ does not act freely and we
 observe that $\Gamma$ does not cover $\Gamma/G$, the graph to the right (the degrees of vertices do not match).  
}

\begin{figure}[ht]
\begin{center}
\mbox{
\centerline{
\includegraphics*[width=1.8in, angle=0]{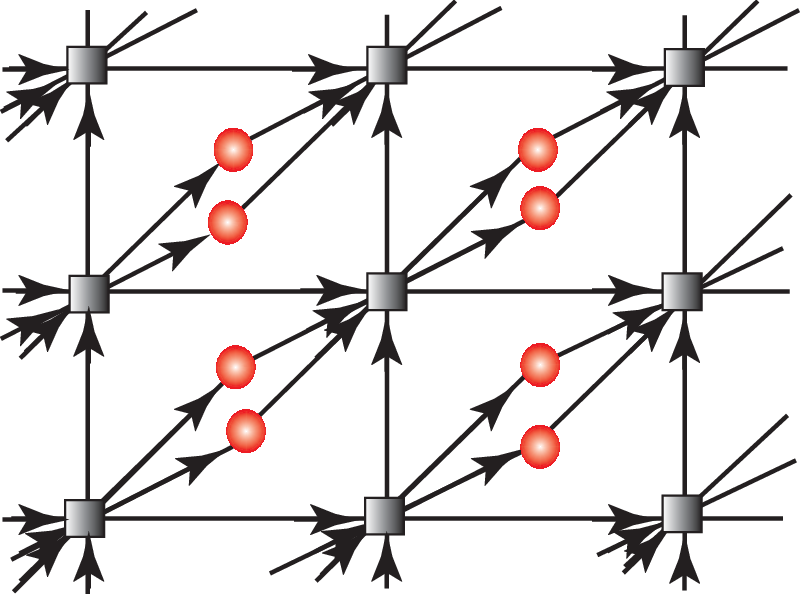}
\hskip 1in
\includegraphics*[width=1.5in, angle=0]{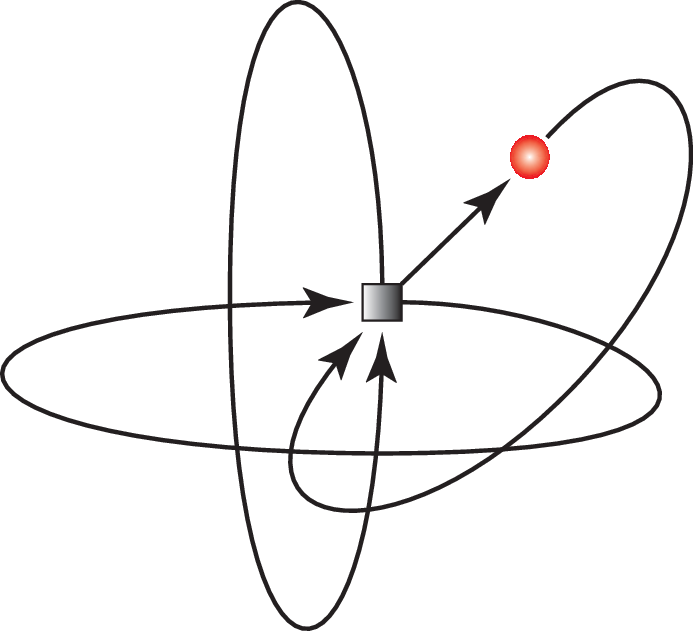}
}
}
\end{center}
\caption{A periodic graph $\Gamma$ with a covering for a group that does not act 
          freely.
 }
\label{nonbary} 
\end{figure}
\end{exam}

A (graph) {\em embedding} is a graph homomorphism that is one-to-one on 
 vertices and edges. Given a subgraph $\Lambda$ of a graph $\Gamma$, 
 the {\em inclusion map} is the embedding of $\Lambda$ into $\Gamma$.

\begin{defn}\label{lifts}
{\rm
Let $\Delta$ be connected and $\phi : \Gamma \to \Delta$ be a covering. For any edge or vertex of $\Delta$ the {\rm fiber of $x$} is the set $\phi^{-1}(x)$, and an element of $\phi^{-1}(x)$ is a {\em lift} of $x$. 
A {\em lift} of $\Delta$ is a connected subgraph $\Lambda\le \Gamma$ of $\Gamma$ such that $\phi$ restricted to the edges of $\Lambda$ is bijective,  inducing a homomorphism 
$\phi|_{\Lambda}:\Lambda\to \Delta$. 
}
\end{defn}

\begin{defn}
{\rm 
Given a graph $\Delta = \langle V, E\rangle$, 
we call an edge set $S \subseteq E$ {\em acyclic} if the subgraph induced by $S$ is acyclic; we say it is a 
 {\em tree set} if $\Delta {|}_S$ is acyclic and connected; and we say it is a {\em spanning tree 
 set} if it is acyclic, connected and contains all vertices of $\Delta$.
 }
\end{defn}

Given a covering $\phi : \Gamma \to \Delta$ 
we say that a lift 
 $\Lambda$ is  {\em proper} if there is an edge set $S_0$ of $\Lambda$ such that $\phi(S_0)$ is a spanning tree set $S$ in $\Delta$ 
 and every edge of $\Lambda$ has its initial vertex in the vertex set of $S_0$.
In this case the vertex set of $S_0$ contains a unique element from 
 every vertex orbit. We call the vertex set of $S_0$ (subset of vertices of $\Lambda$) a set of {\it interior} vertices and denote it with $V_\Lambda^0$, and the subgraph of $\Lambda $ induced by the interior vertices, the {\it interior} of $\Lambda$, denoted $\Lambda^0$.
 The subset of $\Lambda$ 
 that contains the complement of $V_\Lambda^0$, the remaining vertices and edges of $\Lambda$, forms $\partial \Lambda$, called the {\it the boundary of $\Lambda$ 
 with respect to $S_0$}.
The boundary edges have interior initial vertices but not interior terminal vertices. 
A given lift might have several choices for the 
 interior and 
 boundary subgraphs, 
 depending on the choice of $S_0$.

We adapt a bit of lore from \cite{JKM}.

\begin{obse}\label{ProperLift}
Given a connected graph $\Delta$, with a spanning tree $S$, and a covering $\phi : \Gamma \to \Delta$, 
  there exists a proper lift of 
 $\Delta$ that contains a tree set that is a lift of $S$.
\end{obse}

\begin{proof}
A proper lift $\Lambda$ can be constructed as follows. Given a spanning tree set $S \subseteq E_{\Delta}$ we fix a vertex  $v_0 \in V_{\Delta}$ and $v_0^* \in V_{\Gamma}$ in the fiber of 
$v_0$. 
We lift edges and vertices of $\Delta$ inductively  as follows.

Suppose that we have thus far lifted a connected subgraph 
 $\Lambda'$ from $\Delta$ such that $\phi(E_{\Lambda'}) =S' \subseteq S$.
If $S' 
\neq S$,  there exists an edge 
 $e \in S \backslash S'$ that is incident to a 
 vertex in $\phi(V_{\Lambda'})=V_{S'}$.
Let $p=qe$ be a path of minimal length from 
 $v_0$ through $S'$ to an endpoint of $e$ where $q$ is the path through $S'$. We lift $p$
 to a path $p^*=q^*e^*$ that starts at $v_0^*$ such that the lift $q^*$ of $q$ is in $\Lambda'$. This path exists since $\Lambda'$ is a lift of $S'$. Then the edge $e^*$ is a lift of $e$ and $\Lambda'\cup \{e^*\}$ is a lift of $S'\cup \{e\}$.

Inductively, one obtains $\Lambda'$ that is a lift of $S$. 
Then $V_{\Lambda'}$ intersects each orbit of vertices exactly once.
Add to $\Lambda'$ each edge of $\Gamma$ whose initial vertex is 
 in $V_{\Lambda'}$ -- notice that no two such edges are mapped by 
 $\phi$ to the same edge of $\Delta$ - and their incident vertices 
 to obtains $\Lambda$.
As every edge in $\Delta$ has both of its endpoints incident to 
 edges of $S$, $\Lambda$ is a lift of $\Delta$, 
whose interior edges, by construction,  include a tree that is a lift  of $S$. 
\end{proof}

We call  $\Lambda$ described in Observation~\ref{ProperLift} a {\em lift of $\Delta $ with respect to $S$}.

If $\Delta=\Gamma/G$ for some $G\le \Aut(\Gamma)$ acting freely on $\Gamma$, then for every vertex or edge $x$ in $\Delta$, the fiber of $x$ coincides with the orbit of $x$. Let  $\Lambda$ be a proper lift of $\Delta$, then $\Lambda$ contains a single element from every edge orbit.
Thus, we partition the edge set $E_{\Lambda}$ into $\partial 
 E_{\Lambda}$, the set containing edges whose terminal endpoints are in $\partial V_{\Lambda}$, and $E_{\Lambda}^0$, the set containing edges whose both endpoints are in the interior $V_{\Lambda}^0$.
Observe that for any $g \in G$, $g\Lambda$ with vertices 
 $gV_{\Lambda}$ and edges $gE_{\Lambda}$ is also a lift of $\Delta$ (and $g\Lambda$ is proper iff 
$\Lambda$ is), and the set $\{g\Lambda\,|\, g\in G\}$ is the set of lifts that are edge disjoint and whose union is $\Gamma$. Moreover, if $\Lambda$ is proper and  $e\in \partial E_{\Lambda}$
then one endpoint of $e$ is in $V_{\Lambda}^0$ while the other is in $gV_{\Lambda}^0$ for some $g\in G$. In other words, for each $v\in \partial V_{\Lambda}$ there is a $g\in G$ such that $v\in gV_{\Lambda}^0\cap \partial V_{\Lambda}$; thus, $\{ gV_{\Lambda}^0\,|\, g\in G\}$ is a partition of the vertices of $\Gamma$. 

An example of a covering $\phi:\Gamma \to \Delta$ and a lift $\Lambda$ with inclusion map $\xi$ is depicted in Figure~\ref{lift_graphs}. The  map  $\psi:\Lambda\to \Delta$ is the bijection on the edges. Relative to the spanning tree set of red edges of $\Delta$ and $\Lambda$, the interior vertices of $\Lambda$ are indicated with black dots, while the boundary vertices are with unfilled dots. This interior lift is proper.
\begin{figure}[ht]
\begin{center}
$\Delta$
\begin{overpic}
[width = 3.5in]{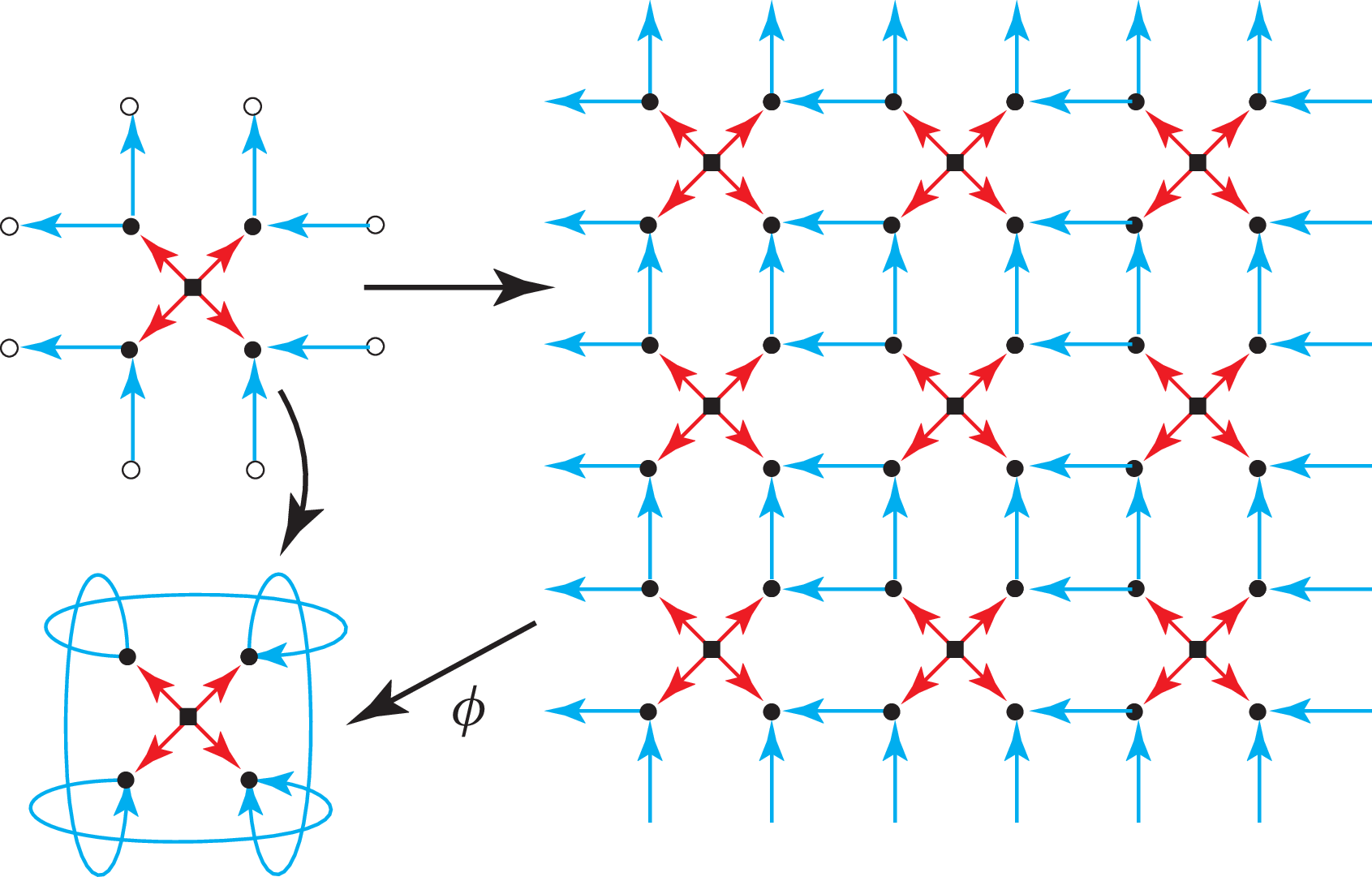}
\put(-5,50){$\Lambda$}
\end{overpic}$\Gamma$
\caption{ A lift graph constructed from the covering $\phi:\Gamma\to \Delta$. The lift graph $\Lambda$ is upper left; the unfilled dots are the boundary  vertices, the black dots are interior. 
}
\label{lift_graphs}
\end{center}
\end{figure}

Suppose $\phi:\Gamma\to\Delta$ is a regular covering and $\Delta$ 
 is connected.
Given a spanning tree set $S$ of $\Delta$, a proper lift $\Lambda$ with respect to $S$ and letting $V_{\Lambda}^0$ be the set of interior vertices of 
 $\Lambda$, the sets $gV_{\Lambda}^0$, 
 $g \in G_{\phi}$ `tile' the vertex set $V_{\Gamma}$ 
 in the sense that every vertex of 
 $\Gamma$ is incident to exactly one lift of $S$.

Let $\phi : \Gamma \to \Delta$ be a regular covering. 
Let $\Lambda$ be a proper lift of $\Delta$ and let $V_{\Lambda}^0$ be 
 the set of $\Lambda$'s interior vertices with respect to some spanning 
 tree set of $\Delta$.
We claim that no two lifts $g_1 V_{\Lambda}^0$, $g_2 V_{\Lambda}^0$ intersect 
 for distinct $g_1, g_2 \in G_{\phi}$.
Towards contradiction, suppose that there exists $v \in gV_{\Lambda}^0 \
 \cap hV_{\Lambda}^0$; then $g^{-1}v, h^{-1}v \in V_{\Lambda}^0$.
As $\phi$ is one-to-one on $V_{\Lambda}^0$ and $\phi(g^{-1}v) = \phi(v) 
 = \phi(h^{-1}v)$, $g^{-1}v = h^{-1}v$.
As $g$ acts freely on $\Gamma$, $g = h$ after all.
Thus $\{ gV_{\Lambda}^0 : g \in G_{\phi} \}$ consists of mutually disjoint 
 sets, and the question becomes: when does it cover $V_{\Gamma}$, i.e., 
 when is it a partition of $V_{\Gamma}$?

\begin{obse}\label{partition}
    Let $\phi : \Gamma \to \Delta$ be a covering and let $\Lambda$ be a proper lift of $\Delta$. Then 
     $\{ gV_{\Lambda}^0 : g \in G_{\phi} \}$ is a partition of $V_{\Gamma}$ 
     if and only if $\phi$ is regular.
\end{obse}

\begin{proof}
Suppose that $\{ gV_{\Lambda}^0 : g \in G_{\phi} \}$ is a partition 
 of $V_{\Gamma}$, and hence covers $V_{\Gamma}$.
If $v \in V_{\Delta}$ and $v_1, v_2 \in \phi^{-1}v$, then for some 
 $g_1, g_2 \in G_{\phi}$, $v_1 \in g_1 V_{\Lambda}^0$ and $v_2 \in g_2 
 V_{\Lambda}^0$.
As $\phi$ is one-to-one on $V_{\Lambda}^0$ and $\phi v_1 = \phi v_2$, 
 $g_1^{-1} v_1 = g_2^{-1} v_2$ and thus $g_2 g_1^{-1} v_1 = v_2$.
The argument extends to edges, and repeating for all such $v$, $v_1$, 
 and $v_2$, $G_{\phi}$ acts transitively on the fibers of $\Delta$, 
 and hence $\phi$ is regular.

Conversely, suppose that $\phi$ is regular; we claim that each vertex of $\Gamma$
 is in some lift $gV_{\Lambda}^0$.
Given $v \in V_{\Gamma}$, as $\phi$ maps $V_{\Lambda}^0$ onto $V_{\Delta}$ and is transitive on $\Gamma$,
 there exists $u \in V_{\Lambda}^0$ such that $\phi u = \phi v$
and there exists $g \in G_{\phi}$ such that 
 $g u = v$.
Hence $v \in gV_{\Lambda}^0$, and we are done.
\end{proof}

As an example, consider Figure~\ref{spanning} which contains the same covering $\phi:\Gamma\to \Delta$ as in Figure~\ref{lift_graphs}.
The figure shows two copies of $\Gamma$,  one to the left and one to the right, while $\Delta $ is depicted with two copies  in the middle. 
Two spanning tree sets of $\Delta$ are indicated with red colored edges in the middle of the figure; the lifts of these spanning 
 tree sets are shown with red colored edges in $\Gamma$. 
Using the methods of \cite{GT} and \cite{Sunada}, can derive $\Gamma$ from 
 $\Delta$ and $\bbbZ^2\le \Aut(\Gamma)$.
 
\begin{figure}[ht]
\begin{center}
\includegraphics*[width = 5in]{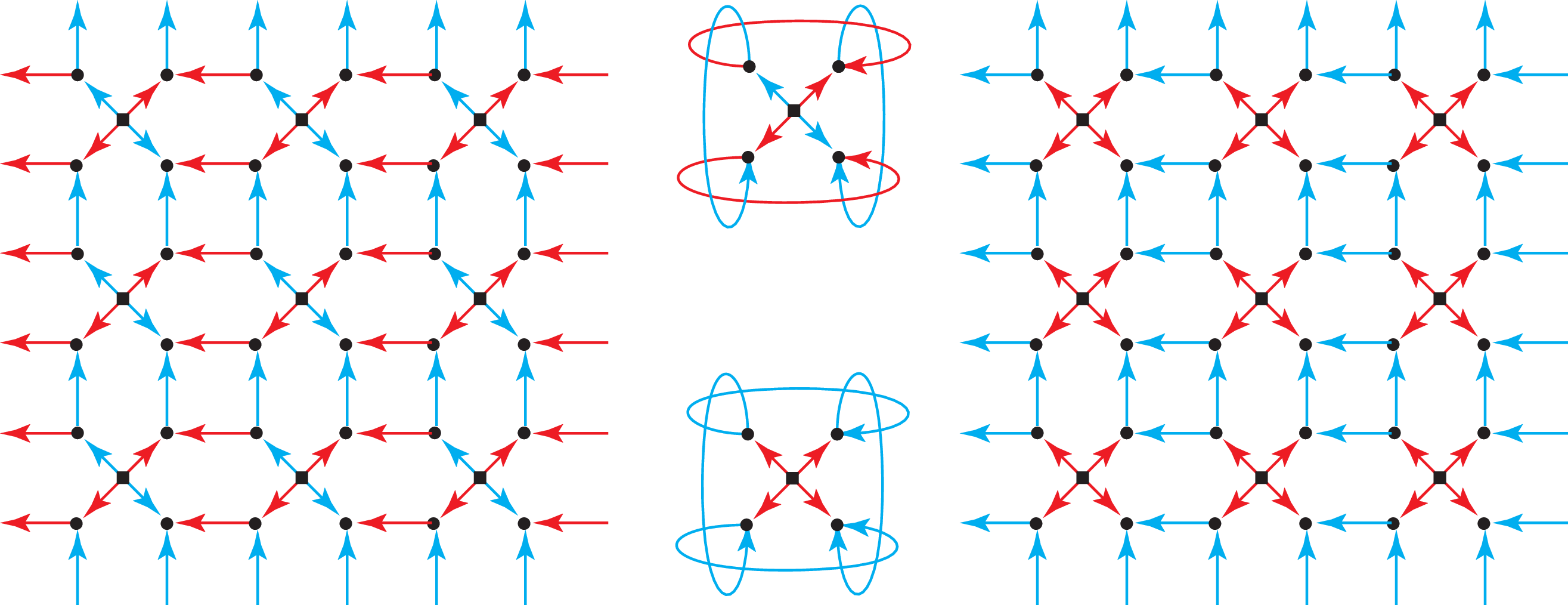}
\caption{Two representations of $\Gamma$ left and right, with two representations of the quotient graph $\Delta$ in the middle.
The middle graphs have different spanning tree sets indicated with red edges. 
The spanning tree set of the upper middle graph is lifted  to the trees of the left depiction of $\Gamma$;  
the lower middle graph spanning tree set is lifted  to the red edge set of $\Gamma$ to the  right.}
\label{spanning}
\end{center}
\end{figure}

Let $\Delta$ be a connected graph with a spanning tree set $S$, and let $v_0 \in 
 V_{\Delta}$.
For each $e \in E_{\Delta} \backslash S$, the {\em return walk of $e$ from $v_0$ 
 via $S$} is $p_e = 
 p_1 e p_2$ where 
 $p_1$ is the unique path in $S$ from $v_0$ to $\iota_{\Delta}(e)$ and $p_2$ is 
 the unique path in $S$  from $\tau_{\Delta}(e)$ to $v_0$.
Similarly, recalling that the reversal of a path $q$ is denoted 
 $\bar q$, the {\em return walk of $e$ from $v_0$ is $\overline{p_e} = \overline{p_2} \bar e \,\overline{p_1}$.}

\begin{defn}\label{tree_voltage}
{\rm
Let $\Delta$ be a connected graph with a spanning tree set $S$.
Given a covering $\phi : \Gamma \to \Delta$, two lifts $S'$ and $S''$ of a spanning 
 tree set $S$ for $\phi$ are {\em adjacent via edge $e$} if $e$ is an edge in $\Gamma$ with one 
 endpoint in $S'$ and the other in $S''$.
 }
\end{defn}

Given a regular covering $\phi : \Gamma \to \Delta$, a vertex $v_0 \in 
 V_{\Delta}$, and two adjacent lifts $S'$ and $S''$ (via  an edge $e$) of a spanning tree set $S$ of $\Delta$, there are lifts $v_0'$ and $v_0''$ of 
 $v_0$ incident to edges of $S'$ and $S''$, respectively,
and there there is a lift $p_e'$ of $p_e$ that is a path from $v_0'$ to $v_0''$.
As $G_{\phi}$ acts freely,  there is
 exactly one $g \in G_{\phi}$ such that $g(S') = S''$, which is the only element of 
 $G$ that satisfies $gv_0' = v_0''$.

The following are a consequence of standard properties of covering spaces, 
 e.g., of \cite[Theorem 5.1]{Sunada}, when $\Gamma$ and $\Delta$ are considered 
 as $1$-complexes.

\begin{lemm}\label{lem:lift-path}
Let $\phi : \Gamma \to \Delta$ be a covering, let $v \in V_{\Gamma}$, and let 
 $p$ be a walk in $\Delta$ whose initial vertex is $\phi(v)$.
Then there exists exactly one walk $q$ in $\Gamma$ whose initial vertex is 
 $v$ and $\phi(q) = p$.
\end{lemm}

Observe that as this lemma applies to all walks in $\Delta$, it applies to 
 initial subwalks.

\begin{prop}\label{prop:upwards}
Given a connected graph $\Gamma$, and finite graphs $\Delta_1$ and 
 $\Delta_2$, with regular coverings $\phi_i : \Gamma \to \Delta_i$, 
 $i = 1, 2$, there exists a connected graph $\Delta_{\top}$, and 
 coverings $\nu : \Gamma \to \Delta_{\top}$ and 
 $\mu_i : \Delta_{\top} \to \Delta_i$, $i =1,2$, such that $\phi_i = 
 \mu_i \nu$.
Moreover,  for any $\Delta$ and coverings $\nu' : \Gamma \to \Delta$ 
 and $\mu_i' : \Delta \to \Delta_i$, $i=1,2$, satisfying $\phi_i = 
 \mu_i' \nu'$, there exists a covering $\psi : \Delta \to \Delta_{\top}$ 
 so that the diagram in Figure~\ref{fig:upwards} commutes.

Furthermore, $|\Delta_{\top}| \leq |\Delta_1| \, |\Delta_2|$ and the 
 cover group of $\nu$ is $G_{\phi_1} \cap G_{\phi_2}$.
\end{prop}

\begin{figure}[h!]
    \begin{tikzcd}
      &  \Gamma 
      \arrow[ldd, bend right, "{\phi_1}" ']
        \arrow[rdd, bend left, "{\phi_2}"] 
        \arrow[dd, bend left=50, "{\nu}" description, dashrightarrow]
        \arrow[d, red,  "{\nu '}" red] 
        & \\
        & \textcolor{red}{\Delta}  
        \arrow[d, red, "{\psi}" near start,dashrightarrow]
        \arrow[rd, red, "{\mu_2'}" red]
        \arrow[ld, red, "{\mu_1'}" ' red]
        &\\
       \Delta_1 &\Delta_{\top} 
       \arrow[r, "{\mu_2}" ',dashrightarrow ]
       \arrow[l, "{\mu_1}",dashrightarrow ]
       &\Delta_2\\
    \end{tikzcd}
   \caption{The diagram for Proposition~\ref{prop:upwards}. The dashed arrows indicate the maps that are asserted by the proposition's statement. Red maps are associated with the second half of the statement.} 
\label{fig:upwards}
\end{figure}
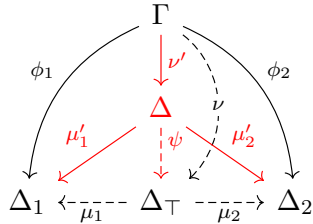

 \begin{proof}
 We construct $\phi : \Gamma \to \Delta_{\top}$ and then verify that 
 it has the desired properties.
For each $i=1,2$, let $G_i = \{ g \in \Aut(\Gamma) : \phi_i g = \phi_i \}$ be the 
 cover group of $\phi_i : \Gamma \to \Delta_i$. Because each $\phi_i$ is 
 regular, by Lemma~\ref{lem:regular}, $\Delta_i \cong \Gamma / G_i$. 
For each vertex or edge $x$ of $\Gamma$ and each $i =1,2$, we denote $[x]_i$ the orbit of $x$ with respect to $G_i$,
and let $[x] = [x]_1 \cap  [x]_2$.
We define $\Delta_{\top}$ with vertices 
$V_{\top} = \{ [v] : v \in V_{\Gamma} \}$ and edges $E_{\top} = \{ [e]
 : e \in E_{\Gamma} \}$, and define $\iota([e]) = [\iota(e)]$ and 
 $\tau([e]) = [\tau(e)]$.
 We observe that $\iota$ and $\tau$ are well-defined on $E_\top$ as follows.
If $[a] = [e]$, then $a, e \in [a]_1\cap [a]_2 = [e]_1\cap [e]_2$, so $[a]_i \cap [e]_i \neq 
 \emptyset$ for $i=1,2$, and therefore 
 $[a]_i = [e]_i$.
Thus, $\iota([a]_i) = \iota([e]_i)$ for $i = 1, 2$, and recalling 
 that $\iota([a]_i) = [\iota(a)]_i$ and $\iota([e]_i) = [\iota(e)]_i$, 
 $[\iota(a)]_i = [\iota(e)]_i$ for $i = 1, 2$, so $[\iota(a)] = [\iota(e)]$, i.e., $\iota([a]) = \iota[(e)]$, and similarly for $\tau$.
The maps $\mu_i:\Delta_\top\to\Delta_i$, where $[x]\mapsto [x]_i$ ($x$ being a vertex or an edge in $\Gamma$), are coverings.
Notice that $|\Delta_{\top}| \leq |\Delta_1| \, |\Delta_2|$ by construction.

The canonical map $\nu:\Gamma \rar \Delta_{\top}$ defined by 
 $\nu(x) = [x]$ for each vertex or edge $x$ of $\Gamma$, where initial and terminal vertices are inherited, is a covering of $\Delta_\top$ such that each $\phi_i$ factors to $\phi_i=\mu_i\nu$.
Indeed, $\nu$ is an onto homomorphism by construction, and for each 
 $v \in V_{\Gamma}$ with two distinct incident edges $a$ and $e$, 
 as $\phi_i$ for $i=1,2$ are coverings, $[a]_i \cap [e]_i = \varnothing$, 
therefore $[a] \cap [e] = \emptyset$, implying that $\nu$ is one-to-one on the incident edges of $v$.

Suppose  $\Delta$ is such that 
 $\nu' : \Gamma \to \Delta$, $\mu_1' : \Delta \to 
 \Delta_1$, and $\mu_2' : \Delta \to \Delta_2$ are coverings satisfying $\phi_i=\mu_i\nu'$.
Define $\psi : \Delta \to \Delta_{\top}$ such that for each edge or vertex $x$ 
 of $\Gamma$,$\psi(\nu'(x)) = [x] =\nu(x)$, and observe that $\psi$ preserves the initial and terminal vertices of each edge $e$: $\psi(\nu'(\iota(e))) = [\iota(e)] = \iota([e]) = \iota(\psi(\nu'(e)))$. Because $\nu'$ is onto and $\nu'(x) = \nu'(y)$ implies  $[x]_i = \mu_i'\nu'(x) = \mu_i'\nu'(y) = [y]_i$ 
 for $i =1,2$, and therefore $[x] = [y]$, $\psi$ is well defined. Because $\nu$ is onto, $\psi$ is an onto homomorphism. Observe that by construction, maps in Figure~\ref{fig:upwards} commute.  
We observe that $\psi$ is a cover:
for each $i=1,2$ and each vertex $v \in V_{\Gamma}$, 
 $\mu_i\phi$ maps $\In(v)$  bijectively onto $\In([v]_i)$, and meanwhile, 
  $\mu_i'\nu'$ maps $\In(v)$  bijectively onto $\In([v]_i)$; as functions in Figure~\ref{fig:upwards} 
  commute by definition, $\psi$ must map $\In(\nu'(v))$  onto $\In([v])$. Similarly for $\Out(v)$.

Finally, we observe that the cover group $G_{\psi}$ is $G_1 \cap G_2$.
On the one hand, if $g \in G_1 \cap G_2$, then for any vertex or edge $x$ of 
  $\Gamma$, $\nu(gx) = [gx] = [gx]_1 \cap [gx]_2 = [x]_1 \cap [x]_2 
  = [x] = \nu(x)$, so $g \in G_{\nu}$.
 On the other hand, if $g \in G_{\nu}$, then $\nu (gx) = 
  \nu(x)$, i.e., $[gx] = [x]$, i.e., $[gx]_i = [x]_i$ for $i = 1, 2$, 
  so for each vertex or edge $x$ of $\Gamma$, there exists $h_{x,i} \in G_i$ 
  for each $i = 1, 2$ such that $gx = h_{x,i} x$; but as $G_i$ acts freely, 
  $h_{x,i} = g$ for each $x$ and $i$, and $g \in G_1 \cap G_2$.
 \end{proof}

Given $\Delta_1$ and $\Delta_2$, we consider a cross product $\Theta =\Delta_1\times \Delta_2$ defined with $V=V_{\Theta}=V_1\times V_2$ and 
$e\in E_\Theta$ with start at $(v_1,v_2)$ terminating at $(v_1,v_2')$ if and only if there are $e_1\in E_1$ starting at  $v_1$ terminating at $v_1'$ and $e_2\in E_2$ starting at $v_2$ and terminating at $v_2'$. 
To ease notation we write $e=(e_1,e_2)$.

\begin{cor}\label{cor:cover-in-product}
    The common cover $\Delta_\top$ identified in Proposition~\ref{prop:upwards} is isomorphic to a subgraph of $\Theta=\Delta_1\times\Delta_2$
\end{cor}
\begin{proof}
By the proof of Prop.~\ref{prop:upwards} the graph $\Delta_\top$ is defined with vertices $V_\top=\{ [v]=[v]_1\cap[v]_2 \,:\, v\in V_{\Gamma}\}$ and edges $E_\top=\{[e]=[e]_1\cap[e]_2\,:\, e\in E_{\Gamma}\}$. Since $\phi_i$ are regular covers, the orbit of a vertex or an  edge $x$ with respect to $G_{\phi_i}$ in $\Gamma$ is precisely the fiber of vertex or an edge $\phi_i(x)$ in $\Delta_i$. We construct a graph $\Delta'$ as a subgraph of $\Theta$ by taking $V_{\Delta'}=\{(\phi_1(v),\phi_2(v))\,:\, [v]_1\cap[v]_2\not = \emptyset, v\in V_\Gamma\} $ and similarly, $E_{\Delta'}=\{(\phi_1(e),\phi_2(e))\,|\, [e]_1\cap[e]_2\not= \emptyset, e\in E_\Gamma\}$. Then the map $[x]\mapsto(\phi_1(x),\phi_2(x))$ iff $[x]=[x]_1\cap[x]_2$ is an isomorphism from $\Delta_\top$ to $\Delta'$. Moreover, the projections $p_i: (\phi_1(x),\phi_2(x)) \mapsto \phi_i(x)$ for $i=1,2$ act precisely as maps $\mu_i$ in the proof of Prop.~\ref{prop:upwards} (Figure~\ref{fig:upwards}).
\end{proof}

We conclude this section with examples illustrating the 
 effect of maps in Figure~\ref{fig:upwards}.

\begin{exam}{\rm 
Two examples illustrating Proposition~\ref{prop:upwards} are shown 
in Figure~\ref{Example}.  In (a) The graph $\Gamma$ is the Cayley graph of 
$\langle      (1, 0), (0, 1) \rangle \cong \bbbZ_4 \times \bbbZ$, the groups associated with coverings $\phi_1$ and $\phi_2$ are $G_{\phi_1} =     \langle (1, 0), (0, 3) \rangle = \bbbZ_4 \times 3\bbbZ$ and $G_{\phi_2} =   \langle (2, 0), (0, 1) \rangle = 2\bbbZ_4 \times \bbbZ$, respectively.  We obtain $G_{\nu} = 2\bbbZ_4 \times 3\bbbZ$, and $G_{\mu_1} \cong     \bbbZ_2$ while $G_{\mu_2} \cong \bbbZ_3$.

\begin{figure}[ht]
\begin{center}
(a)
\includegraphics*[width = 2.7in]{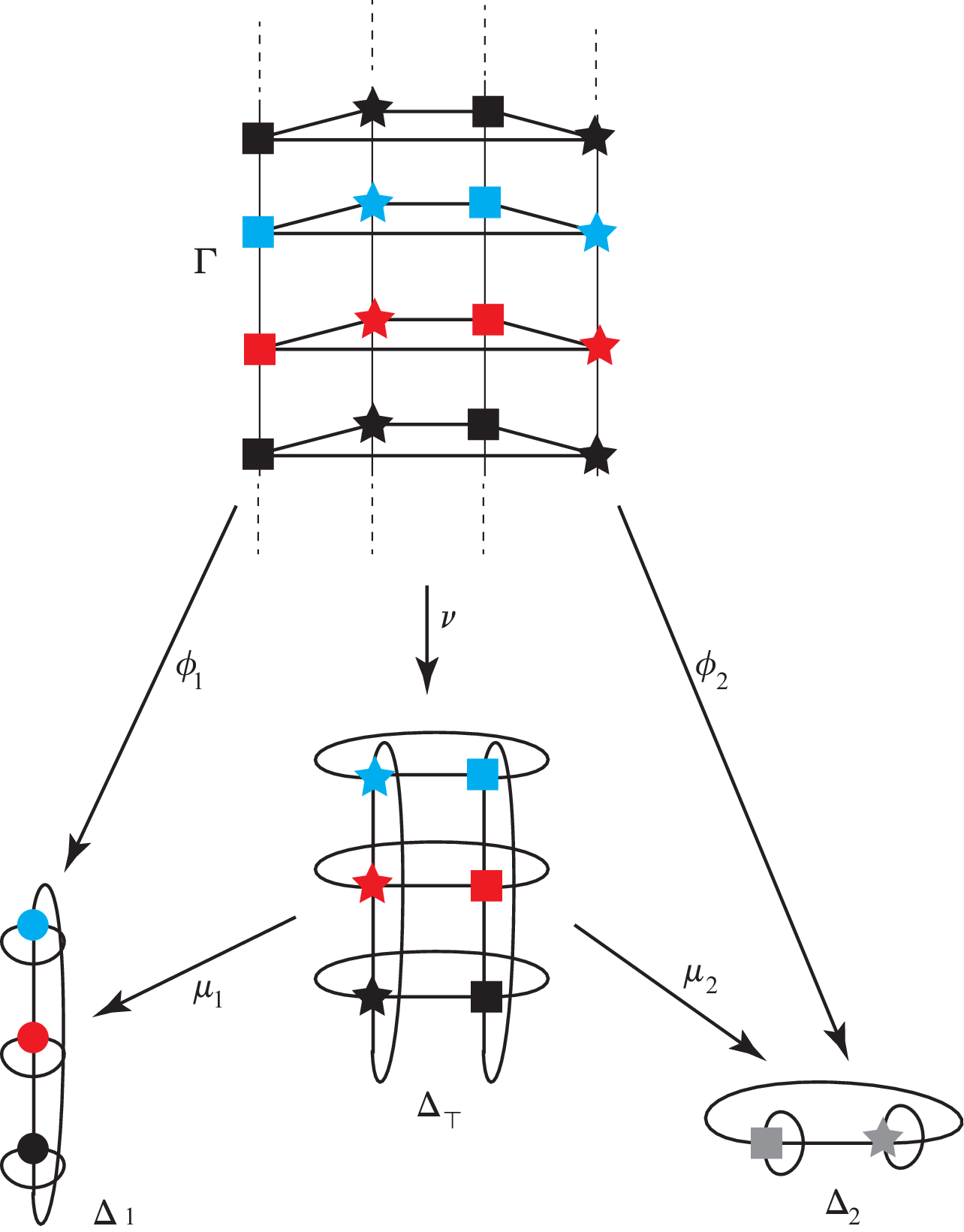} \qquad
(b) \includegraphics*[width = 1.7in]{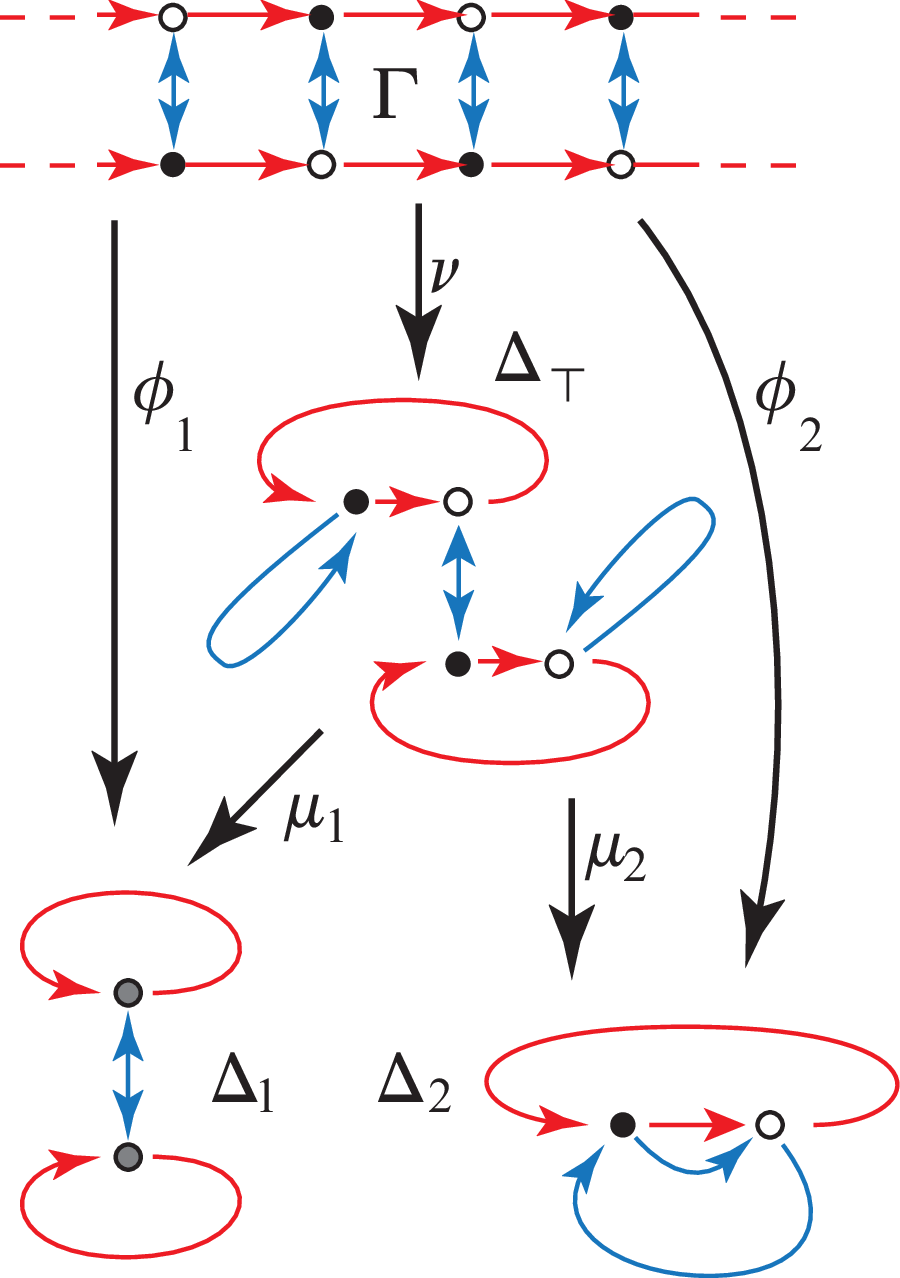} 
\caption{Two illustrations of the situation in 
Propositon~\ref{prop:upwards}. (a) The cover maps are depicted by setting vertex color to map to the same vertex color for covering $\Delta_1$,
and vertex shape maps to the same vertex shape
          for covering $\Delta_2$.
(b) The same color edges map to the respective color edges in $\Delta_1$ and $\Delta_2$. Vertical neighboring vertices map to vertices in $\Delta_1$ while horizontal neighboring vertices map to $\Delta_2$, white vertex to white and black vertex to black. 
In both examples 
         $\Delta_{\top}$ is the minimal common cover graph.
}\label{Example}
\end{center}
\end{figure}
In Figure~\ref{Example}(b) 
the vertices of $\Gamma$ are elements of $\mb Z \times \mb Z /2\mb Z$.
Let $f(i, j) = (i, j + 1 \mod 2)$ (so that $2 f$ is the identity, i.e, $f$ 
 flips the vertices up-down), and let  $g(i, j) = (i + 1, j)$ ($g$ shifts 
 to the right, so that $g^k(i, j) = (i + k, j)$ for any $k \in \bbbZ$) be
 two elements of $G=\Aut(\Gamma)$.
Consider the composition of a shift and a flip $h = f g\in G$ such that 
 $h^k(i, j) = (i + k, j + k \mod 2)$ for $k \in \mb Z$.
The two quotient graphs are $\Delta_1=\Gamma / \langle g \rangle$ and 
 $\Delta_2=\Gamma / \langle h \rangle$ shown.
Since $2 g = 2 h$, there is a common cover $\Delta_\top$ of $\Delta_1$ and 
 $\Delta_2$ as depicted. 
 }
\end{exam}

\section{Voltages}\label{Svoltages}

In this section we review some notions of voltage graphs from \cite[\S 2]{GT}, define lifting of voltages  and prove several properties of lifted voltages and the derived respective graphs.

\begin{defn}\label{derived-graph-a}
{\rm
Let $\Delta$ be a graph, and let $G$ be a group.
A {\em voltage assignment} on edges of
$\Delta$ is a function $\gamma : E_{\Delta} \to G$.
A {\em voltage graph} is an ordered pair $(\Delta, \gamma)$, where $\gamma$ is 
 a voltage assignment on $E_{\Delta}$.
The {\em voltage group} of $\gamma$ is the subgroup of $G$ generated with $\gamma(E_\Delta)$ and denoted $\langle \gamma(E_{\Delta})\rangle$. 

Given the voltage graph $(\Delta, \gamma)$, the {\em derived graph of $\Delta$ 
 by $\gamma$}, denoted  $\Delta \times_\gamma G$, is the graph $\Gamma = ( V, E)$ defined by:
  \[
 V = V_\Delta \times G,\quad E = E_\Delta \times G,
  \]
 and
  \[
 \iota(e, g) = (\iota(e), g), \quad
 \tau(e, g) = (\tau(e), g \gamma(e)).
  \]
  }
\end{defn}

From the definition, it follows that the derived graph $\Delta \times_\gamma G$ is a cover of $\Delta$ under the {\it canonical map} $(x,g)\mapsto x$ for a vertex or an edge $x$ in $\Delta$. Each $(v,g)$ (resp., $(e,g)$) is lift of vertex $v$ (resp., edge $e$). We observe that the derived graph depends on the voltage assignment.

\begin{exam}\label{DifferentVoltages}
{\rm The voltage assignments, although coming from the same group,  may produce distinct derived graphs.
Consider the graph $\Delta$ with four vertices and five edges depicted at the middle of 
 Figure~\ref{SameQuotient}, and let the voltage group be $\bbbZ$.
The three black edges form a spanning tree, so let $\gamma(\mbox 
 {\rm black}) = 0$. 
We consider two assignments for the red and blue edge. In the first case, 
let $\gamma(\mbox {\rm red}) = 1$ and $\gamma(\mbox 
 {\rm blue}) = -1$, while in the second case we set 
 $\gamma(\mbox {\rm red}) = 1$ 
 and $\gamma(\mbox {\rm blue}) = 2$. The first case has a derived graph depicted on the top while the second case has the derived graph at the bottom. The derived graphs are not isomorphic because the bottom graph has cycles of length 4 while the top has no cycles of length 4.}
\end{exam}

\begin{figure}[ht]
\centerline{
\includegraphics*[width = 3in]{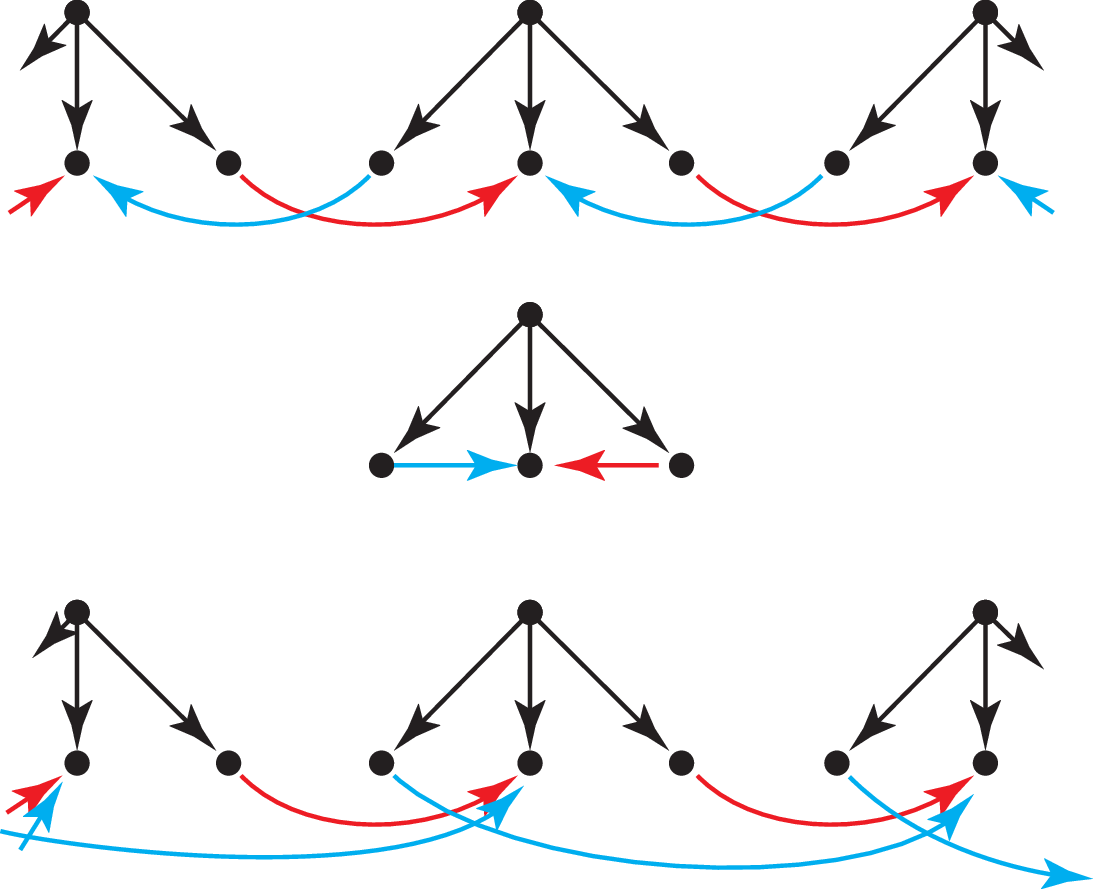} }
\caption{{\small Two non-isomorphic derived graphs (top and bottom) obtained from the same graph (middle) with two distinct voltage assignments from $\bbbZ$. In both cases black edges have assignments $0$, and the red edge has assignment 1. The blue edges has one assignment $-1$ (top derived graph) and $2$ (bottom derived graph).} }
\label{SameQuotient}
\end{figure}

Let $G$ act on $\Delta \ttimes {\gamma} G$ as follows.
For $h\in G$, let $h(v, g) = (v, hg)$ and $h(e, g) = (e, hg)$ so that 
 $h(\iota(e, g)) = h((\iota(e), g)) = (\iota(e), hg) = \iota(e, hg)$ and 
 $h(\tau(e, g)) = h((\tau(e), g\gamma(e)) = (\tau(e), hg\gamma(e)) =
 \tau(e, hg)$.

Consider a cyclic walk $p=v_0 e_1^{\epsilon_1} v_1 
\cdots v_{n-1} 
 e_n^{\epsilon_n} v_n$ in $\Delta$ where $v_n = v_0$ and $\epsilon_i = 1$ if $p$ traverses $e_i$ along its orientation (i.e., $v_{i-1} = \iota(e_i)$ and $\tau(e_i) = v_i$) or $\epsilon_i = -1$ is $p$ traverses $e_i$ against its orientation (i.e., $v_{i-1} = \tau(e_i)$ and $\iota(e_i) = v_i$) for each $i=1,\ldots,n$. 
 We set $\gamma(p) = \prod_{k=1}^n \gamma(e_k)^{\epsilon_k}$.
Then for every $g\in G$, as 
 ${}^g\gamma(p)(v_0, g) = (v_0, {}^g \gamma(p) g) = (v_0, g \gamma(p))$,  we can regard $\gamma(p)$ as moving   $(v_0, g)$  
 to $(v_0, g \gamma(p))$.\footnote{Here, the {\em left conjugate} of $\gamma(p)$ by $g$ is ${}^g\gamma(p) = g\gamma(p)g^{-1}$} This implies the following.

 \begin{rema}\label{rem:cyclic-walk}
    With the canonical covering,  a cyclic walk $p$  in $\Delta$ lifts to a cyclic walk $p'$ in $\Delta\times_\gamma G$ if and only if $\gamma(p)=1.$
 \end{rema}

The voltage assignments may have non-identity voltages scattered about 
 in the voltage graph. It is suitable to have 
 edges of trivial voltage correspond to the interior of a lift subgraph, i.e., the 
 edges of trivial voltage induce a connected subgraph of $\Delta$.
To do this, we ``condense'' the voltage assignment by combining \cite[Theorems 2.1.2 
 \& 2.5.4]{GT} as follows.

\begin{defn}\label{def:condensation}
{\rm
Let $(\Delta, \gamma)$ be a pointed voltage graph with $v_0$ as a base, and  let $S$ be a spanning tree of 
 $\Delta$.
Let $\gamma_S(e)=1_G$ for $e\in S$ and when 
$e \in E_{\Delta} \backslash 
 S$, let $\gamma_S(e)=\gamma(p_e)$ where $p_e$ 
 is a return walk of $e$ via $S$ from $v_0$.
 Then   
$\gamma_S$  is
the {\em condensation assignment} (or just {\em condensation}) of $\gamma$ with respect to $S$.
If a voltage assignment is such that  every edge of some spanning tree set is assigned the identity then we say it is {\em condensed}.
}
\end{defn}

\begin{defn}\label{def:voltage-assignment}
{\rm
Let $\Gamma$ be connected and  $\phi : \Gamma \to \Delta$ be a covering. 
Fix a proper lift $\Lambda$ of $\Delta$ and call it the {\em reference
 lift}.
A {\em voltage assignment with respect to $\Lambda$} 
on $\Delta$ is the following:
\begin{itemize}
 \item
  If $e \in E_{\Lambda}^0$, set $\gamma(e_\Delta) = 1_G$
 \item
     If $e \in \partial E_{\Lambda}$, with  one end point of $e$  in $\Lambda^0$ and the other endpoint in $ gV_{\Lambda}^0$ for some
     $g \in G_{\phi}$, set $\gamma(e_\Delta) =  g$.
\end{itemize}
}
\end{defn}

We will observe that when $\Delta$ has the voltage assignment condensed with 
 respect to some eligible lift $\Lambda$, the derived graph of $\Delta$ 
 by $\gamma$ is isomorphic to $\Gamma$.
Recall from Observation~\ref{ProperLift} that a proper lift $\Lambda$ of 
 $\Delta$ has at least one lift $S'$ of a spanning tree set $S$ of $\Delta$ 
 in its interior, i.e., the endpoints of each edge of $S'$ are interior 
 points of $\Lambda$.
Thus, a voltage assignment that is condensed with respect to $\Lambda$ 
 assigns the voltage $1$ to every edge of $S$.
However, the edges assigned voltage $1$ will be those covered by the 
 interior edges of $\Lambda$, which may include edges not in $S$.

\medskip
\noindent{\it Notation.} In the rest of the sections we consider a covering $\phi : \Gamma \to \Delta$ and  for a vertex or an edge $x$ of $\Gamma$ for $\phi(x)$, the image of $x$, we use notation $x_\Delta$.

\begin{lemm}\label{lem:voltage-assignment}
   Let  $\phi : \Gamma \to \Delta$ be a regular covering, and $\Gamma$ be connected. Let $\Lambda$ be a proper lift of $\Delta$ in $\Gamma$ and $\gamma$ voltage assignment on $\Delta$ with respect to $\Lambda$, i.e., 
   $\gamma(e) = 1$ if $e \in E_{\Lambda}^0$.
   Then  $\Gamma \cong \Delta \times_\gamma \langle \gamma(E_\Delta)\rangle$.
\end{lemm}

\begin{proof}
 First, we show that $G_\phi\cong \langle\gamma(E_\Delta)\rangle$. Since voltages in $\Delta$ are from $G_\phi$ we have $\langle\gamma(E_\Delta)\rangle\le G_\phi$.
 
On the other hand, suppose $g\in G_\phi$ and consider the lift $g\Lambda$ in $\Gamma$.
As $\Gamma$ is connected, and $\{ g\Lambda^0 : g \in G_{\phi} \}$ is a 
 partition of the vertices of $\Gamma$, for any interior vertex $v$ of $\Lambda$ and any interior vertex $gw$ of $g\Lambda$, there is a path in $\Gamma$ from $v$ to $gw\in g\Lambda_0$. 
Then $g\in \langle\gamma(E_\Delta)\rangle$ follows from two observations: 
 first, the path $p$ is a composition of paths that visit adjacent proper lifts of $\Delta$: $\Lambda = \Lambda_0,\Lambda_1, \ldots, \Lambda_k=g\Lambda$ and second, every path that starts in  a lift $\Lambda_i$ and terminates in a lift $\Lambda_{i+1}$ maps to a path in $\Delta$ whose voltage assignment is in $\langle\gamma(E_\Delta)\rangle$. 
 That is, if $p$ is a path from $\Lambda_i=h_i\Lambda$ to $\Lambda_{i+1}=h_{i+1}\Lambda$ that traverses an edge $e_i$ whose one endpoint is in $h_i\Lambda^0$ and the other in $h_{i+1}\Lambda^0$, then there is an edge $e$ in 
 $\partial \Lambda$ with one endpoint in $V_\Lambda ^0$ and another in $h_i^{-1}h_{i+1}V_\Lambda^0$ whose image $e_\Delta$ in $\Delta$ has voltage $\gamma(e_\Delta)=h_i^{-1}h_{i+1}$ for $1=0,\ldots,k-1$ with $h_0=1_G$. 
 Since, by the definition of $\gamma$, every path that visits only vertices in $V_\Lambda^0$ has edges mapped in $\Delta$ with voltages equal to $1_G$, the path $p$
 in $\Gamma$ maps to $\phi(p)=p_\Delta$ with $\gamma(p_\Delta)=h_1h_1^{-1}h_2\cdots h_{k-1}h_{k-1}^{-1}h_k=g$, and 
 hence $g\in \langle\gamma(E_\Delta)\rangle$.

We now verify that $\Gamma \cong \Delta \ttimes {\gamma} G_{\phi}$.
For each $v\in V_\Lambda^0$ and $g \in G_{\phi}$, let $\varphi(gv) = 
 (\phi(v), g)$; as $G_{\phi}$ acts freely on $\Gamma$, $\varphi$ is well-defined.
Similarly, for each $e\in E_\Gamma$ with $\iota(e) \in V_{\Lambda}^0$, and 
 each $g\in G_\phi$, let $\varphi(ge) = (\iota(e)_\Delta, g)$; again, as 
 $G_{\phi}$ acts freely on $E_{\Gamma}$, $\varphi$ is well-defined. Since $\phi$ is regular, the sets $gV_\Lambda^0$ partition $V_\Gamma$, so this map is bijective. 
We claim that the edge-vertex incidence relation is preserved. 
Then $\gamma(\phi(e)) = g^{-1}h$.
Since $\varphi(e) = (\phi(e), g)$, $\varphi(\iota(e)) = (\phi(\iota(e)), g) 
 = \iota(\phi(e), g)$ and similarly, $\varphi(\tau(e)) = (\phi(\tau(e)), h) 
 = (\phi(\tau(e)), g\gamma(\phi(e))) = \tau(e, g)$.
%
 %
 Thus, the edges are preserved and the map is a graph isomorphism.
 \end{proof}

From Def.~\ref{def:voltage-assignment} and Lemma~\ref{lem:voltage-assignment} we have the following observations.

\begin{obse}\label{obse:condensed}
Let $(\Delta, \gamma)$ be a connected voltage graph with a  
 voltage assignment such that for a spanning tree set $S \subseteq E_{\Delta}$, 
 $e \in S \implies \gamma(e) = 1_G$, where $G = \langle \gamma( E_{\Delta} )\rangle$. By the covering  $\Delta \times_\gamma G\to \Delta$, there is a proper lift of $\Delta$, the reference lift $\Lambda$,  such that $\gamma$ is a voltage assignment with respect to $\Lambda$. Then the interior vertices of $\Lambda$ can be taken to be $(v, 1_G)$, and all incident edges with two  endpoints in $\Lambda^0$  map to edges with voltage assignment $1_G$. Thus $\Lambda$ contains a lift $S_0$ of $S$, {\it the reference lift of $S$}, as a subgraph. 
All  vertices of every lift $gS_0$ of $S$ 
in $\Delta \ttimes {\gamma} G$ belong to the interior of the same  lift $g\Lambda$ of $\Delta$, and hence
for any two vertices $u,v\in V_\Delta$ we can consider $(u,g),(v,g)\in V_{g\Lambda}=gV_\Lambda$.
Moreover,  if $(e,g)$ and $(e',g)$ are edges in $\partial g\Lambda$ whose one endpoint is in $gS$ and another in $g'S$, for $g\not =g'$, then $\gamma(e_\Delta)=\gamma(e'_\Delta)=g^{-1}g'$. In other words, all edges connecting two adjacent lifts, are lifts of edges in $\Delta$ with the same voltage assignment.
\end{obse}


%
%
%
 

\begin{exam}
   {\rm
Consider the voltage graph 
$(\Delta,\gamma)$ at upper left in Figure~\ref{condense} with voltage assignment in $\mb Z$: red edges have voltage $+1$,  blue edges have 
voltage $-1$, and black edges have voltage $0$.
 At upper right is the resulting derived graph
 $\Delta\times_\gamma \mb Z$.
 Although the upper left voltage graph is connected, the derived graph has two components. At lower left  is a condensation of the voltage graph 
 $(\Delta,\gamma_S)$,  in which the orange edge gets voltage $2$ and the green edge gets voltage $-2$.  The edges of the spanning tree set $S$ that produces the condensation
 are labeled with $0$. 
In this case the group generated by the voltages is $2\bbbZ$, and using that group we obtain the (connected!) derived graph $\Delta \times_{\gamma_S} 2\mb Z$ at lower 
 right.

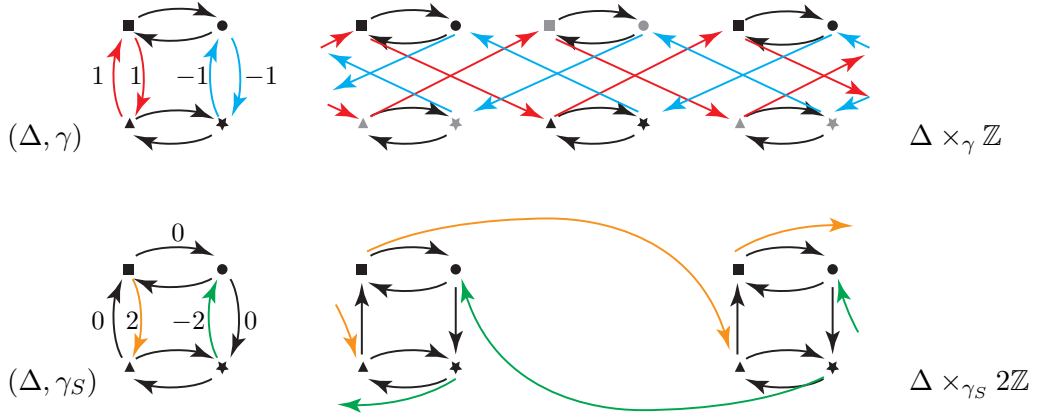
\begin{figure}[ht]
\centerline{
\begin{overpic}[width = 4in]{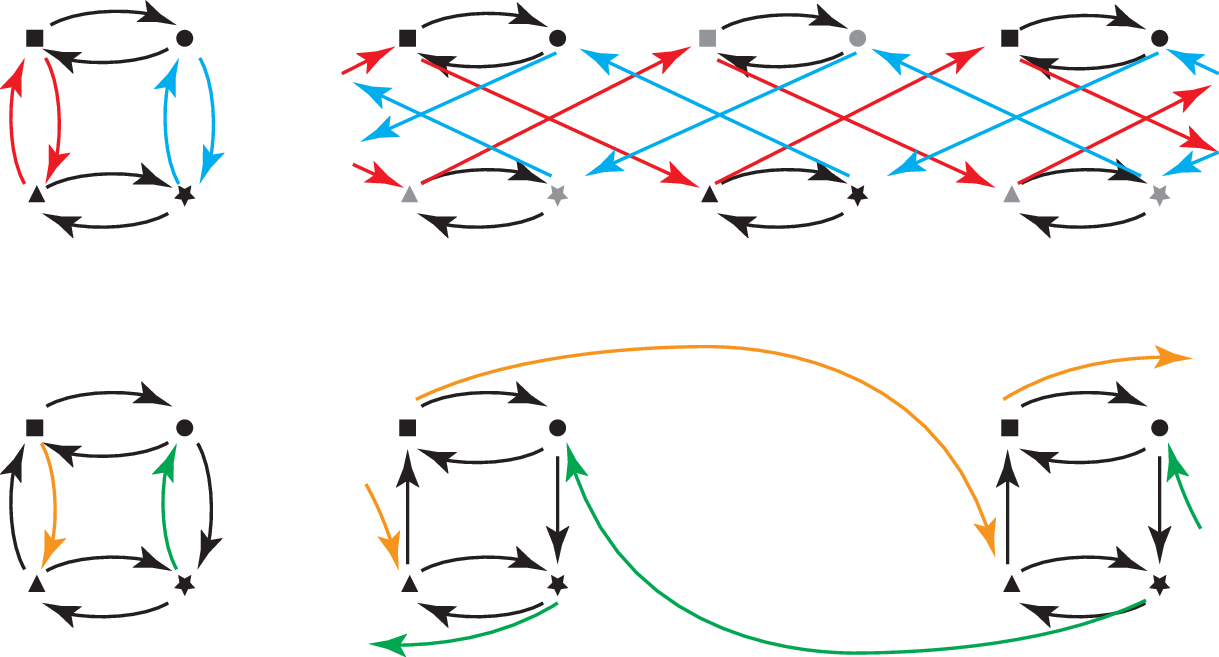} 
\put(-13,35){$(\Delta,\gamma)$}
\put(-2,43){{\small {$1$}}}
\put(3,43){{\small{$1$}}}
\put(9,43){{\small{$-1$}}}
\put(18,43){{\small{$-1$}}}
\put(105,35){$\Delta \times_\gamma \mb Z$}
\put(-13,3){$(\Delta,\gamma_S)$}
\put(-2,11){{\small {$0$}}}
\put(2.5,11){{\small{$2$}}}
\put(8.5,11){{\small{$-2$}}}
\put(8.5,22.5){{\small{$0$}}}
\put(18,11){{\small{$0$}}}
\put(105,3){$\Delta \times_{\gamma_S} 2\mb Z$}
\end{overpic}}
\caption{{\small A voltage graph $(\Delta, \gamma)$ with voltages from $\mb Z$ at top left and a condensation $(\Delta, \gamma_S)$ with respect of a spanning tree set with edges labeled $0$ generating $2\mb Z$ at left bottom. The corresponding derived graphs are to the right. The upper graph $\Delta \times _\gamma \mb Z$ has two components one of which is isomorphic to $\Delta\times_{\gamma_S} 2\mb Z$, the graph at the bottom. 
}}
\label{condense}
\end{figure}
}
\end{exam}
We combine \cite[Theorem 2.5.4]{GT} and \cite{AG};

\begin{theo}\label{2.5.4}
Let $\Delta$ be connected, let $\gamma : E_{\Delta} \to G$ be a voltage 
 assignment, and let $\gamma'$ be a condensation of $\gamma$.
Then $\Delta \ttimes {\gamma'} G$ consists of $[G : \langle 
 \gamma'(E_{\Delta}) \rangle]$ copies of the sole component of $\Delta 
 \ttimes {\gamma} \langle \gamma'(E_{\Delta}) \rangle$.
\end{theo}

We note that if a voltage graph $(\Delta, \gamma)$ is connected and 
 if its voltage assignment is condensed, then  $\Delta \ttimes {\gamma} G$ is connected, where  $G$ is the group generated 
 by its voltages.
From \cite{AG} (see \cite[Theorem 2.5.1]{GT}), it follows that unless $G = G'$, 
 $\Delta_{\top} \ttimes {\gamma'} G$ is not connected.
 We have the following proposition.

\begin{prop}\label{prop:condensed_voltages}
Let $(\Delta, \gamma)$ be a connected voltage graph and let $G$ be the group 
 generated from the voltages assigned to $\Delta$.
For any connected component $\Gamma_0$ of $\Delta \ttimes {\gamma} G$ and any 
 spanning tree set $S$ of $\Delta$, the condensation $\gamma_S$ satisfies 
 $\Gamma_0 \cong \Delta \ttimes {\gamma_S} \langle \gamma_S(E_\Delta) 
 \rangle$.
\end{prop}

\begin{proof}
Let $\Gamma = \Delta \ttimes {\gamma} G$ and let $\phi : \Gamma \to \Delta$ be 
 the canonical covering. Choose $(v_0, g_0) \in \Gamma_0$ and a spanning tree set $S$ in $\Delta$.
We set $G'=\langle \gamma_S(E_\Delta)\rangle=\langle 
 \gamma_S(E_{\Delta}\backslash S)\rangle$. Because $\Delta$ is connected, 
 each component of $\Gamma$ covers $\Delta$, and so does $\Gamma_0$. 
Using Theorem~\ref{2.5.4}, without loss of generality, suppose that $g_0 = 
 1_G$.
Denote vertices and edges of $\Gamma_0$ with $V_0$ and $E_0$ resp., and the reference lift of $S$ in $\Gamma_0$ 
by $S_0$, so that for each 
 $g \in G$, $gS_0$ is the lift of $S$ (to $\Delta \ttimes {\gamma} G$) 
 containing $(v_0, g)$.
Let $H = \{ h \in G : (v_0, h) \in V_0 \}$, and we claim that $H\le G$.
\begin{itemize}
 \item $H$ is a set of automorphisms of $\Gamma_0$.
	If $g \in H$ and $(v, h) \in V_0$ then $g(v,h)=(v,gh)$, since $g$ maps a 
     path $p$ from $(v_0, 1_G)$ to $(v, h)$ to a path $gp$, which goes 
     from $(v_0, g)$ to $(v, gh)$.
	Thus $(v, gh) \in V_0$, so the restriction of $g$ to $\Gamma_0$ maps $V_0$ 
	 to $V_0$ (and thus $E_0$ to $E_0$).
 \item $H$ is closed under composition.
	If $g, h \in H$ and $p$ is a path in $\Gamma_0$ from $(v_0, 1_G)$ to 
	 $(v_0, h)$, then $g$ maps $p$ to a path $gp$ from $(v_0, g)$ to 
	 $(v_0, gh)$.
	As $(v_0, g) \in V_0$, so is $(v_0, gh)$, and $gh \in H$.
 \item $H$ is closed under inversion.
	If $g \in H$ and $p$ is a path in $\Gamma_0$ from $(v_0, 1_G)$ to 
	 $(v_0, g)$, then $g^{-1}$ maps $p$ to a path $g^{-1}p$ from $(v_0, g^{-1})$ 
	 to $(v_0, 1_G)$, so $(v_0, g^{-1}) \in V_0$ and $g^{-1} \in H$.
\end{itemize}

We  verify that $H = G' $.

First, we claim that $H \leq G'$.
If $h \in H$, then there is a path $p$ from $(v_0, 1_G)$ to $(v_0, h)$. 
We prove by induction on the number of lifts that $p$ visits that $h \in G'$.
If $p$ stays within a single lift of $v_0$ then $p$ ends at the vertex 
 $(v_0, 1_G)$, and $1_G\in G'$.
For the inductive step, suppose for all paths $p'$ from $(v_0, 1_G)$ to 
 $(v_0, h')$ that visit at most $k-1$ lifts of $v_0$ in $\Gamma_0$ we have $h'\in G'$, and let $p$ be a path from $(v_0, 1_G)$ to $(v_0, h)$ that visits $k$ lifts of $v_0$.
Then $p=p'q$ where $p'$ is a path
 from $(v_0, 1_G)$ to a lift containing $(v_0, h')$ visiting $k-1$ lifts of $v_0$ and $q$ is a path from that lift containing $(v_0, h')$
 to a lift containing $(v_0, h)$ that, besides these two vertices, does not visit another lift of $v_0$.
Thus $q$ runs from a lift to an adjacent lift so that $q_{\Delta}$ contains only one edge not in $S$, i.e., for some edge $e$ in $\Delta$, $q_{\Delta} = re^{\varepsilon}t$ where $r$ and $s$ are paths in $S$ while $e$ is not in $S$, and such that $r$'s initial vertex is $(v_0,h')$, the terminal vertex of $p'$ and $t$'s terminal vertex is $(v_0,h)$.
Then $\gamma(q_{\Delta}) = \gamma(r e^{\varepsilon} t) = \gamma(r) \gamma(e)^{\varepsilon} \gamma(t) = \gamma(e)^{\varepsilon}$ as $r$ and $t$ run through a lift of $S$.
Thus, as $\gamma(p_{\Delta}') \in G'$, $h = \gamma(p_{\Delta}) = \gamma(p_{\Delta}') \gamma(q_{\Delta}) = \gamma(p_{\Delta}') \gamma(e)^{\varepsilon} \in G'$.

%
 
Second, to confirm that $G' \leq H$, observe that 
for each $e_\Delta\in E_\Delta\setminus S$, the lift of $p_{e_\Delta}$, the return path of $e_\Delta$ via $S$, that starts at  $(v_0,1_G)$ has a terminal vertex at $(v_0,g)$ where $g=\gamma_S(e_\Delta)=\gamma(p_{e_\Delta})$. Since $H$ is closed under composition, $G'\le H$.

Finally, recalling that $\Gamma_0$ is the component of $\Delta \ttimes 
 {\gamma} G$ containing $(v_0, 1_G)$, we define $\varphi : \Gamma_0 \to 
 \Delta \ttimes {\gamma_S} G'$ as follows. Let $h\mapsto \hat h$ be an isomorphism from $H$ to $G'$, and 
let $V_S$ be the set of vertices incident to the reference lift $S_0$.
Observe that the sets $\hat hS_0$, $\hat h \in G'$, partition the vertices in $\Gamma_0$. We 
  set $\varphi (v_\Delta,h) = (v_\Delta,\hat h)$ for every vertex 
 $(v_\Delta, h)$ in $hV_S$.
Similarly, let $E_S$ be the set of vertices whose initial endpoints are 
 in $V_S$ (observe that $S \subseteq E_S)$, and set $\varphi(e_\Delta,h) = 
 (e_\Delta,\hat h)$ for $(e_\Delta,h)\in hS_0$.
By Observation~\ref{obse:condensed}, $\varphi$ maps lifts of 
 $S$ to lifts of $S$, implying the map is one-to-one,
 and $\varphi$ is onto as every vertex and every edge is in 
 a lift.
We see  that $\varphi$ preserves the incidence relation between edges and 
 endpoints: for an edge $e_\Delta$ in $\Delta$ we have $\varphi(\iota(e_\Delta, h)) 
 = \varphi(\iota(e_\Delta), h) = (\iota(e_\Delta), \hat h)=\iota(e_\Delta,\hat h)=\iota(\varphi(e_\Delta,h))$.
 And similarly for $\tau$.
\end{proof}

Observe that Proposition~\ref{prop:condensed_voltages} holds for every choice of a spanning set $S$.
Notice also that if $\Gamma$ is connected and $\phi : \Gamma \to \Delta$ is
 a covering and $S$ a fixed spanning tree set  of $\Delta$ and its reference lift $S_0$, then $\gamma_S$ is
  voltage assignment with respect to $S_0$ on $\Delta$, then $G_{\phi} = \langle \gamma_S(e) :  e \in S^c \rangle$.
We are interested in good covers as defined below.

\begin{defn}\label{def:good-cover}
{\rm
Consider a voltage graph $(\Delta,\gamma)$ and a covering $\mu:\Delta_\top\to \Delta$.
We say that $\mu$ is {\em good covering}  and $\Delta_\top$ a {\em good cover} if, for 
 each cyclic walk $p$ in $\Delta$ such that $\gamma(p) = 1_{G}$, 
every lift of $p$ with respect to $\mu$ in $\Delta_\top$ is a cyclic walk in $\Delta_\top$.
}
\end{defn}

Consider $\Gamma=\Delta\times_\gamma G$ with the canonical covering  $\phi:\Gamma\to \Delta$  and let $v_0\in V_\Gamma$ be a base vertex with $\phi(v_0)=v_\Delta\in V_\Delta$. Let $\Delta_\top$ be a connected cover of $\Delta$  with $\mu:\Delta_\top\to\Delta$. Let $v_\top$ be the vertex in $\Delta_\top$ such that $\mu(v_\top) 
 = v_{\Delta}$.

Define $\nu:\Gamma\to \Delta_\top$ as follows. We set  $\nu(v_0)=v_\top$ and for each edge $e$ starting at $v_0$ we define $\nu(e)$
to be the  edge $f$ in $\Delta_\top$ that starts at $v_\top$
 such that $\phi(e)=\mu(e)$. By Lemma~\ref{lem:lift-path}, $f$ is unique. 
 Inductively, for each vertex $v$ in $\Gamma$ and path $p$ that starts at $v_0$ and terminates at $v$, define $\nu(v)$ to be the 
 terminal vertex of the unique path $p_\top$ in $\Delta_\top$ that starts at $v_\top$ such that $\phi(p)=\mu(p_\top)$. By construction, $\phi=\mu\nu$,
 but $\nu$ is a covering if and only if it is well-defined.

\begin{prop}\label{prop:good-cover}
The map $\nu : \Gamma \to \Delta_\top$ is covering if and only if $\Delta_\top$ is a good  cover.
\end{prop}

\begin{proof}
Notice that it suffices to restrict attention to cyclic walks with initial 
 and terminal vertex $v_{\Delta}$ such that $\gamma(p) = 1_{G}$.

Suppose that every cyclic walk $p$ in $\Delta$ of initial 
 and terminal vertex $v_{\Delta}$ with $\gamma(p) = 1_{G}$ lifts 
 to a cyclic walk in $\Delta_\top$.
We first see that  $\nu$ is well-defined. Suppose there are two walks $p$ and $q$  from $v_0$ to a 
 vertex $v$ in $\Gamma$.
Then as $p \bar q$ is a cyclic walk in $\Gamma$ and  $\phi (p \bar q)$ is a cyclic walk in 
 $\Delta$. As $p \bar q$ starts and ends at $v_\Delta$, by Remark~\ref{rem:cyclic-walk}, $\gamma_i (p\bar q) 
 = 1_{G_i}$.
Thus, $p \bar q$ lifts to a (unique, by Lemma~\ref{lem:lift-path}) cyclic walk from $v_{\top}$ to 
 $v_{\top}$ and therefore $\nu(p)=\nu(q)$.

To see that $\nu$ is covering, notice that by construction, $\nu$ is a 
 homomorphism and $\mu\nu=\phi$. As $\phi$ maps vertex figures of $\Gamma$ bijectively onto vertex 
 figures of $\Delta$, $\mu$ being covering implies  that vertex figures from $\Delta$ are lifted 
 bijectively to vertex figures in $\Delta_\top$, $\nu$ maps vertex figures of $\Gamma$  
 bijectively to vertex figures of $\Delta_\top$.
We see that $\nu$ is onto,  for any $u \in V_\top$, let $p_\top$ be a path in $\Delta_\top$ 
 from $v_\top$ to $u$; the path $p_\top$ exists because $\Delta_\top$ is connected. The walk $\mu( p_\top)$ lifts to a walk 
 $p$ in $\Gamma$ from $v_0$ to a vertex $u'$.
Then $\nu (p)$ = $p_\top$ and $\nu(u') = u$. Note that $p$ must be a path if $p_\top$ is a path, because every cycle in $\Gamma$ maps to a cycle in $\Delta$ with voltage $1$, which lifts to a cycle in $\Delta_\top$.

Conversely, suppose that $\nu$ as defined above is a covering. 
Assume that there exists a vertex $v \in V_\Delta$ and paths $p$ 
 and $q$ from $v_\Delta$ to $v$ such that $\mu$ lifts the cyclic path 
 $\phi(p \bar q)$ to a walk 
 $p_\top \overline{q_\top}$ in $\Delta_{\top}$ from $v_{\top}$ to a vertex $w  \neq v_\top$. Then by definition of $\nu$, $\nu(v_0)=v_\top$, and $\nu(v_0)=\nu(v_0p\bar q)=w$, which is a contradiction with $\nu$ being a covering, and hence, well defined. 
\end{proof}

\begin{defn}\label{lift_voltage}
{\rm
Let $\mu : \Delta_{\top} \to \Delta$ be a covering and  $\gamma : E_{\Delta} 
 \to G$ a voltage assignment.
We say that the voltage assignment  $\gamma' : E_{\Delta_{\top}} \to G : e \to \gamma(\mu(e))$ is the  {\em lift of $\gamma$}.
}
\end{defn}
 
We use the following theorem \cite[Theorem 5.5]{Sunada} to prove our next result. 

\begin{theo}\label{thm:sunada}
Let $\nu:\Gamma \to \Delta_\top$ and $\mu:\Delta_\top \to \Delta$  coverings such that $\phi=\mu\nu$. If $\phi$ is regular, then so is  $\nu$. The group  $G_\nu$ is a normal subgroup of $G_\phi$ and $G_\mu \cong G_\phi/G_\nu$ if and only if $\mu$ is regular. 
\end{theo}

We are now ready for the primary result of this section, which is a  
 prerequisite for the results in Section~\ref{LVI}.

\begin{prop}\label{Lift_voltages-new}
Let $\Gamma$, $\Delta_\top$ and $\Delta$ be connected and let $\nu : \Gamma \to \Delta_{\top}$ and $\mu : \Delta_{\top} \to \Delta$ be regular coverings, so that $\phi = \mu \circ \nu : \Gamma \to \Delta$ is also a regular covering.
Let  $\gamma : E_{\Delta} \to G_{\phi}$ be a condensed voltage 
 assignment.
If $\gamma_\top$ is a condensation of a lift of $\gamma$ in $\Delta_\top$ with respect to some spanning tree set 
in $\Delta_\top$, 
 then $\Gamma \cong \Delta_{\top} 
 \ttimes {\gamma_\top} \langle \gamma_\top (E_{\Delta_\top}) \rangle$. 
\end{prop}

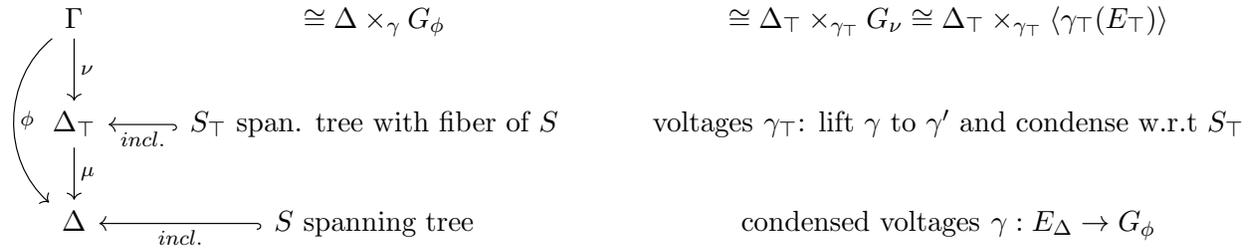
\begin{figure}[ht]
   \begin{tikzcd}
        \Gamma \arrow[d,"{\nu}"] 
        \arrow[dd, bend right = 50, "{\phi}"] &
         \cong \Delta \times_\gamma G_\phi & 
         \cong \Delta_\top \times_{\gamma_\top} G_\nu \cong\Delta_\top\times_{\gamma_\top}\langle \gamma_\top(E_\top)\rangle
        \\
       %
        \Delta_{\top} \arrow[d, "{\mu}"] 
        & S_\top \textrm{ span. tree with fiber of $S$}\arrow [l,"{incl.}", hook] 
       & \textrm{voltages $\gamma_\top$: lift $\gamma$ to $\gamma'$ and condense w.r.t $S_\top$}\\
        \Delta & S \textrm{ spanning tree}\arrow [l,"{incl.}", hook]  & \textrm{condensed voltages $\gamma: E_\Delta \to G_\phi$}\\
    \end{tikzcd}
\caption{Situation of Prop.~\ref{Lift_voltages-new} and its proof.}
\label{componsition}
\end{figure}

\begin{proof}
Let $S$ be a spanning tree set of $\Delta$. By 
 Observation~\ref{ProperLift}, there exists a proper lift $\Lambda_0$ of 
 $\Delta$ in $\Delta_{\top}$ containing a tree set $S_{0}$ that is a 
 lift of $S$ (Figure~\ref{componsition}).
As $G_{\mu}$ is regular, by Lemma~\ref{partition}, the union of the lifts $g S_{0}$, 
 $g \in G_{\mu}$, is the edge set of a spanning forest $F_{\top}$; by 
 judiciously adding edges joining the trees together without forming any 
 cycles, one obtains spanning tree set $S_{\top}$ of $\Delta_{\top}$.

Similarly, having $S_{\top}$ as a spanning tree set of $\Delta_{\top}$, 
 there is a proper lift $\Lambda_{\top}$ of $\Delta_{\top}$ in $\Gamma$ with respect to $S_\top$, and $\Lambda _\top$ contains a lift of $\Delta$, called $\Lambda_\Delta$ containing $S_\Delta$, a lift of $S$ (Figure~\ref{fig:covers}).
We choose an interior vertex $v_0$ of $\Lambda_{\Delta}$ as a 
 reference vertex of $\Gamma$.
Let  $v_{\top} = \nu (v_0)$ be the reference vertex of $\Delta_{\top}$ and $v_{\Delta} =\mu (v_{\top}) = \phi (v_0)$ the  base vertex of $\Delta$.
As with $\mu$, as $\nu$ is regular, by Lemma~\ref{partition},
the sets $hV_\top^0$, $h \in G_{\nu}$ partition the set $V_{\Gamma}$.

Since $G_\nu$ is a subgroup of $G_\phi$ and the voltages of $\Delta$ generate $G_\phi$, there are two possibilities for each edge $e \in E_{\Delta}$: either $h=\gamma(e_\Delta) \in G_\nu$ or $g=\gamma(e_\Delta)\not \in G_\nu$ (see Fig.~\ref{fig:covers}).
As $h\in G_\nu$, $h\Lambda_\Delta$ is a lift within $h\Lambda_\top$, and $\nu (h\Lambda_\top)=\nu (\Lambda_\top)=\Delta_\top$, hence $\nu (h\Lambda_\Delta)=\Lambda_0$.
On the other hand, if $g\not \in G_\nu$, then $\nu g\not = \nu$
and $g\Lambda_\top \not = \Lambda_{\top}$, so that these two lifts do not share interior vertices. However, $g\Lambda_\Delta$ is a lift of $\Delta$ under $\phi=\mu\nu$, and treating $G_{\mu}$ as $G_{\phi} / G_{\nu}$, 
$\nu(g\Lambda_\Delta)$ is $g G_{\nu} \Lambda_0$, a lift of $\Delta$ in $\Delta_\top$.

We obtain a voltage assignment $\gamma_{\top}$ from $\gamma$.
Assign voltages $\gamma':E_{\Delta_\top}\to G_\phi$ by lifting voltages 
 of $\Delta$ to $\Delta_\top$ according to $\mu$, i.e.,  $e\in 
 E_{\Delta_\top}$, $\gamma'(e)=\gamma(\mu(e))$.

Observe that every component of $\Delta_{\top} \ttimes {\gamma'} G_{\phi}$ is isomorphic to $\Delta \ttimes {\gamma} G_{\phi}$.
By Theorem~\ref{2.5.4}, the components of $\Delta_{\top} \ttimes {\gamma'} G_{\phi}$  are isomorphic, so it suffices to prove the isomorphism for the component $\Gamma_{\top}$ containing $(v_{\top}, 1)$, where $1$ is the identity of $G_{\phi}$.

For any vertex $(v, g)$ in $\Gamma_{\top}$, there is a path $p$ from $(v_{\top}, 1)$ to $(v, g)$, and the covering $\Gamma_{\top} \to \Delta_{\top} : (u, h) \mapsto u$ maps $p$ to a walk $p_{\top}$ in $\Delta_{\top}$ from $v_{\top}$ to $v$.
Let $p_{\top} = s_0 e_1^{\varepsilon(1)} s_1 e_2^{\varepsilon(1)} s_3 \cdots s_{n1} e_n^{\varepsilon(n)} s_n$, where $e_i$ is a lift of an edge not in $S$, and each $s_i$ is a lift of $S$,  $\varepsilon_i = 1$ if $s_{i-1}=\iota(e_i)$. Then 
 $\gamma'(p_\top) = \prod_{i\le n} \gamma'(e_i)^{\varepsilon(i)}$.
By induction on the length of $p$, we find that if $(u_j, h_j)$ is the terminal vertex of $s_j$, then $h_j = \prod_{i<j} \gamma'(e_i)^{\varepsilon(i)}$ and thus $\gamma'(p) = g$.
So $\Gamma_\top \to \Delta \ttimes {\gamma} G_{\phi} : (v, g) \mapsto (v, g): (e, g) \mapsto (e, g)$ is an isomorphism.

By Proposition~\ref{prop:condensed_voltages}, 
for the condensation $\gamma_{\top}$ of $\gamma'$ with respect to $S_{\top}$ we have  that  
$\Delta_{\top} \ttimes {\gamma_{\top}} \langle \gamma_{\top}(E_{\top}) \rangle$ is isomorphic to any component of $\Delta_{\top} \times_{\gamma'} G_{\phi}$ and hence to $\Gamma_\top\equiv \Gamma$.
\end{proof}

\begin{figure}[h]
\centerline{
\begin{overpic}
[width = 3.5in]{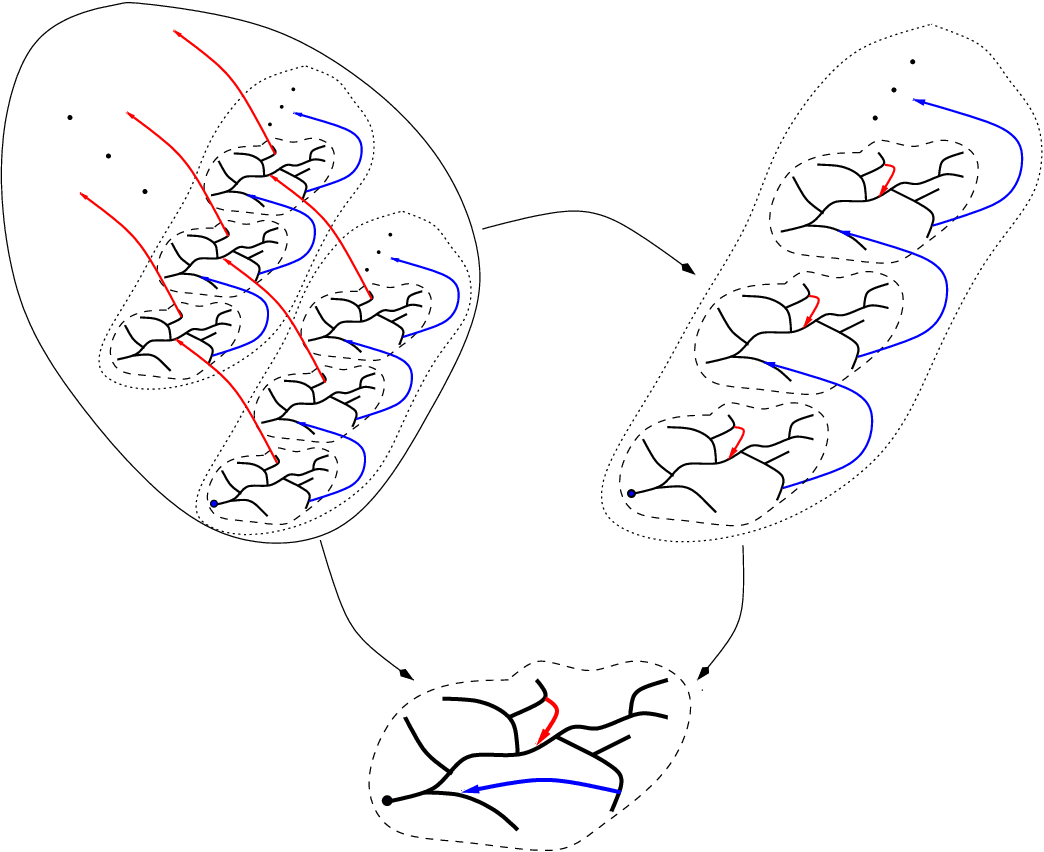} 
\put(-10,65){$\Gamma$}\put(22,75){{\small{$h\Lambda_\top$}}}
\put(30,0){$\Delta$}\put(36,7){{\small{$v_\Delta$}}}
\put(37,62){$\Lambda_\top$}
\put(62,42){$\Lambda_0$} \put(58,32){\small{$v_\top$}}
\put(53,63){$\nu$}\put(73,23){$\mu$}\put(30,20){$\phi$}
\put(20,39){{\Small{$\Lambda_\Delta$}}}
\put(17,30){{\Small{$v_0$}}}
\put(95,50){$\Delta_\top$} \put(90,56){{\small{$g$}}}\put(84,40){{\small{$g$}}}
 \put(50,5){{\small{$g$}}}\put(54,13){{\small{$h$}}}
\end{overpic}}
\caption{Situation in Proposition~\ref{Lift_voltages-new}, blue edges have voltages $g\in G_\phi\setminus G_\nu$; red edges have voltages $h\in G_\nu$}	  
\label{fig:covers}
\end{figure}
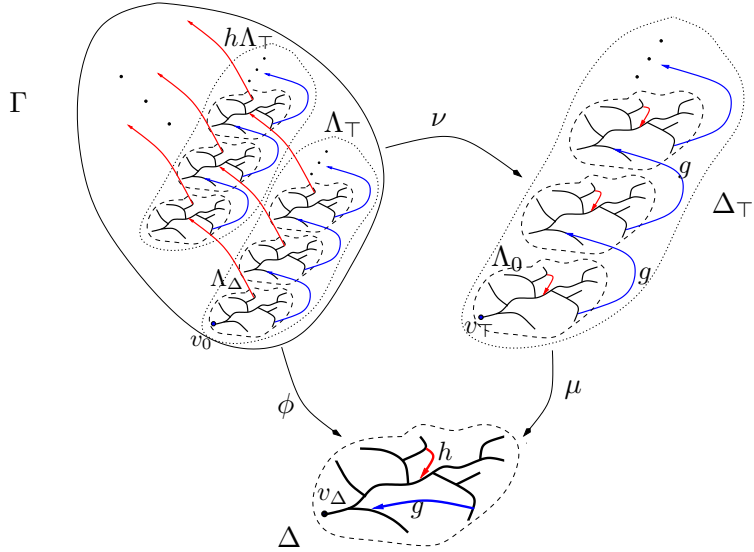

We overview the situation in Prop.~\ref{Lift_voltages-new}.
    The voltages $\gamma':E_{\Delta_\top}\to G_\phi$ are assigned by lifting voltages of $\Delta$ to $\Delta_\top$ according to $\mu$, i.e.,  $e\in E_{\Delta_\top}$, $\gamma'(e)=\gamma(\mu(e))$. 
    Then edges with voltages in $G_\nu$ are in the interior of $\Lambda_0$ as well as in interiors of other lifts of $\Delta$ within its fiber in $\Delta_\top$. 
    While edges with voltages $g\in G_\phi\setminus G_\nu$ are edges in $\partial \Lambda_0$ (as well as other boundaries of lifts of copies of $\Lambda_0$)  connecting $\Lambda_0$ with $g\Lambda_0$. 
    Moreover, all edges connecting vertices from $\Lambda_0$ to $g\Lambda_0$ must have edges with voltages $g\in G_\phi\setminus G_\nu$ since all interior edges are with voltages $1_G$ or $h\in G_\nu$. Similarly, if there is an edge with voltage $g$ connecting a vertex from $g_1\Lambda_0$ with $g_2\Lambda_0$ for some $g, g_1$ and $g_2$, then $gg_1=g_2$.
    Observe that $\Lambda_0$ and all $g\Lambda_0$ 
    have lifts of a spanning tree set $S$ in $\Delta$  whose edges have voltages $1_G$. 
    Let $S_0$ be the lift of $S$ in $\Lambda_0$ and $v_\top$ in $S_0$ a lift of $v_0$. We can  form a spanning tree set $S_\top$ as a subgraph of $\Delta_\top$ containing the lifts of $S$ 
    (i.e., $\cup gS_0$),  adding edges from $\partial \Lambda_0$ and its copies that connect adjacent lifts of $\Delta$. We add one edge for each pair of adjacent lifts such that the resulting set forms a spanning tree set $S_\top$. 
    We assign values $1_G$ to those edges, and perform condensation $\gamma_\top$ with respect to $S_\top$ and the base vertex $v_\top$. The condensation changes voltages.

    If $e$ is in interior of $g\Lambda_0$ then its voltage changed from $h\in G_\nu$ to $g^{-1}hg$ which is in $G_\nu$ since $G_\nu$ is normal in $G_\phi$. On the other side, if an edge $e'$ is in 
    $\partial g\Lambda_0$, adjacent to $\Lambda_0$, following Observation~\ref{obse:condensed}, its voltage changes from $g$
    to $gg^{-1}=1_G$. This is because if $\gamma'(e)=g\not\in G_\nu$ connecting adjacent lifts $S_0$ and $gS_0$, the return walk of $e$ from $v_\top$ via $S_\top$ traverses a path $p_1ep_2e'p_3$ where $p_1$ and $p_3$ are paths within $S_0$ from $v_\top$ to the $\iota(e)$, and $\tau(e')$  to $v_\top$ in $S_0$, respectively, with $e'\in S_\top$ while $p_2$ is within $gS_0$ from $\tau(e)$ to $\iota(e')$. 
    Then by Observation~\ref{obse:condensed}, $g=\gamma'(e)=\gamma'(e')^{-1}$. Inductively, the situation is the same for voltages of all edges joining adjacent lifts of $S_0$. The partitions of graphs defined with covers $\phi,\nu,\mu$ are illustrated in Figure~\ref{fig:covers}.

\section{Isomorphic Covers}\label{LVI}

In this section we prove our main result showing the following. Given two finite connected graphs with voltage assignment that are compatible in a sense to be defined below, they generate isomorphic derived graphs if and only if there is a common finite cover of the two graphs such that the lifts of the voltage assignments induce and isomorphism between their respective groups. 

We recall the following, from \cite{Skoviera, KL0}.

\begin{theo}\label{thm:iso-voltage0}
Let $\Delta$ be a connected finite graph with condensed voltage assignments $\gamma_i : E_{\Delta} \to G_i$ $(i=1,2)$ with respect to the same spanning tree set $S$. 
Let $\Gamma_i=\Delta \ttimes {\gamma_i} G_i$, $i =1,2$, be their derived graphs and 
 $\phi_i: \Gamma_i 
 \to \Delta$ be the 
 corresponding coverings.
There exists an isomorphism $\varphi : \Gamma_1 \to \Gamma_2$ 
  such that $\phi_2\varphi=\phi_1$ if and only if 
 the map on voltage assignments $\gamma_1(e) \mapsto \gamma_2(e)$, $e \in E_{\Delta}$, induces a group 
 isomorphism of $G_1$ onto $G_2$.
\end{theo}

We outline the proof for the convenience of the reader.

\begin{proof}
If the assignment $\gamma_1(e)\mapsto \gamma_2(e)$ for $e\in E_\Delta$ 
 extends to a group isomorphism $F : G_1 \to G_2$ then the map 
 $\Delta \ttimes {\gamma_1} G_1 \to \Delta \ttimes {\gamma_2} G_2$ defined 
 by $(v,g)\mapsto (v,F(g))$ and $(e, g) \mapsto (e, F(g))$ extends to a 
 graph isomorphism: $\varphi(\iota(e, g)) = \varphi((\iota(e), g)) = 
 (\iota(e), F(g)) = \iota(e, F(g)) = \iota(F(e, g)) = \iota(\varphi(e, g))$, 
 and similarly $\varphi \tau = \tau \varphi$.

Conversely, suppose that there exists a label preserving isomorphism 
 $\varphi : \Gamma_1 \to \Gamma_2$ such that $\phi_2\varphi=\phi_1$. Fix a base vertex $v_0 \in V_{\Delta}$ and call $(v_0,1_{G_i})$ the base vertices in $\Gamma_i=\Delta \ttimes {\gamma_i} G_i$ ($i=1,2$). 
 Without loss of generality,  we assume  that $\varphi$ maps the base vertex to the base vertex: $\varphi(v_0, 1_{G_1}) = \varphi(v_0, 1_{G_2})$.
 (If $\varphi(v_0, 1_{G_1}) = (v_0, g_2)$ with $g_2 \neq 1_{G_2}$, then as $\kappa : (v, g) \mapsto (v, g_2^{-1} g) : (e, h) \mapsto (e, g_2^{-1} h)$ is an automorphism of $\Delta \ttimes {\gamma_2} G_2$, we can take $\kappa \varphi$ instead of $\varphi$.)
Let $\Lambda_1$ and $\Lambda_2$ be the respective proper lifts of $\Delta$ in $\Gamma_1$ and $\Gamma_2$ that contain lifts of $S$ and the base vertices in $\Gamma_1$ and $\Gamma_2$. Then by Def.~\ref{def:voltage-assignment}, $\gamma_i$ is the voltage assignment of $\Delta$ with respect to $\Lambda_i$ ($i=1,2$),  and $\varphi(\Lambda_1)=\Lambda_2$. For every edge
 in the boundary of $\Lambda_i$, i.e., $(e,1_{G_i})\in\partial \Lambda_i$, $i=1,2$,
 its image $e_\Delta$ in $\Delta$ has a non-empty voltage assignment in $\Delta$. 
 In fact, by Prop.~\ref{prop:condensed_voltages}, $G_i$ is generated by these voltage assignments. 

We observe that the correspondence produces a bijection on generators. Let $F$ be the function from generators of $G_1$ to generators of $G_2$ induced by the correspondence $\gamma_1(e) \mapsto \gamma_2(e)$.

First, we claim that $F$ is well-defined and one-to-one on the generators: $\gamma_1(e) = \gamma_1(f)$ iff $\gamma_2(e) = \gamma_2(f)$.
To see that $F$ is well-defined, first observethat a cyclic path $p$ in $\Gamma_1$ at vertex $(v_0,1_{G_1})$ maps to a cyclic path $q$ in $\Gamma_2$ at vertex $(v_0,1_{G_2})$, both being lifts of the same cyclic path $p_\Delta$ at vertex $v_0$ in $\Delta$.
Then suppose there are two edges $e_\Delta,f_\Delta$ with non-identity voltages in $\Delta$ with $\gamma_1(e_\Delta)=\gamma_1(f_\Delta)=g$ while 
$h_1=\gamma_2(e_\Delta)\not =\gamma_2(f_\Delta)=h_2$. Then $e_\Delta$ and $f_\Delta$ lift in $\Gamma_1$ to edges $e_1,f_1\in \partial\Lambda_1$ that have endpoints in $g\Lambda_1$, while the same edges lift in $\Gamma_2$ to $e_2,f_2\in \partial \Lambda_2$ such that $e_2$ ends in $h_1\Lambda_2$ and $f_2$ ends in $h_2\Lambda_2\not = h_1\Lambda_2$. In particular, $h_2^{-1}h_1\Lambda_2\not=\Lambda_2$. But this contradicts 
the fact that $\varphi:\Gamma_1\to\Gamma_2$ is an isomorphism, because 
the cyclic path $p_{e_\Delta}\overline{p_{f_\Delta}}$ where $p_{e_\Delta}$ and $\overline{p_{f_\Delta}}$ are the respective return walks of $e_\Delta$ and the reverse of $f_\Delta$ from $v_0$ via $S$, lifts to a cyclic path $p$ that starts and ends in $(v_0,1_{G_1})$ in $\Gamma_1$, but  lifts to a path $q$ that is not a cyclic path in $\Gamma_2$, while $\phi_2\varphi=\phi_1$, it must be $\varphi(p)=q$. Thus $F$ is well-defined. A symmetric argument establishes that $F$ is one-to-one: 
reversing $\Lambda_1$ and $\Lambda_2$ shows that for all edges $e_\Delta$ in $\Delta$, $\gamma_2(e_\Delta)=\gamma_2(f_\Delta)$ implies $\gamma_1(e_\Delta)=\gamma_1(f_\Delta)$.
Observe that for any edge $e$ in $\Delta$, if $e$ is covered by an interior edge of $\Lambda_i$, then $\gamma_i(e) = 1_{G_i}$, and as $\varphi$ preserves interior edges, $\gamma_1(e) = 1_{G_1}$ iff $\gamma_2(e) = 1_{G_2}$.

As the sets of generators are finite and $F$ is one-to-one, $F$ is bijective on the generators, and is thus onto.  To see that $F$ extends to a homomorphism,  it suffices to verify that products of generators are mapped to corresponding products of generators.
If $g,g'$ are voltages in $G_1$ of $e_\Delta,e_\Delta'$ respectively, and $p_{e_\Delta},p_{e_\Delta'}$ are return walks from $v_0$ via $S$ in $\Delta$, respectively, then $p_{e_\Delta}p_{e_\Delta'}$ lifts to a path from $(v_0,1)$ to
$(v_0,gg')$ in $\Gamma_1$, and $(v_0,1)$ to $(v_0, hh')$ in $\Gamma_2$ where $\gamma_2(e_\Delta)=h=F(g)$, and $\gamma_2(e_\Delta')=h'=F(g')$. Remark~\ref{rem:cyclic-walk} implies that a product of generators that is $1_{G_1}$ maps into a product of generators that is $1_{G_2}$.
\end{proof} 


Let $\Gamma_1$ and $\Gamma_2$ be graphs and let $G_1 \leq \Aut(\Gamma_1)$ 
 and $G_2 \leq \Aut(\Gamma_2)$ act freely on $\Gamma_1$ and $\Gamma_2$, 
 resp, and let $\Delta_i=\Gamma_i/G_i$.
Suppose that there is an isomorphism $\psi : \Gamma_1/G_1 \to \Gamma_2 / G_2$, i.e., $\Delta_1\cong \Delta_2$.  Let $S_1$ and $S_2$ be  spanning trees of $\Gamma_1/G_1$ and 
 $\Gamma_2/G_2$, resp., such that $\psi (S_1) = S_2$. 
Let $\gamma_1$ and $\gamma_2$ be the condensed voltage assignments induced 
 by $S_1$ and $S_2$, resp.
Then by Theorem~\ref{thm:iso-voltage0} above, $\Gamma_1 \cong \Gamma_2$ iff there exists an automorphism $\alpha$ of $\Delta_1=\Gamma_1 / G_1$ fixing $S_1$ such 
 that the correspondence between voltages induced by $\psi\alpha$ extends to 
 an isomorphism of $G_1$ to $G_2$.

We turn to the case when $\Delta_1\not \cong \Delta_2$. 

\begin{exam}\label{notfree}
{\rm Recall the example in Figure~\ref{fig:upwards}(b). The voltage assignment of the graphs are depicted in 
Figure~\ref{not_free}.

\begin{figure}[ht]
\centerline{
\begin{overpic}
[width = 2.5in]{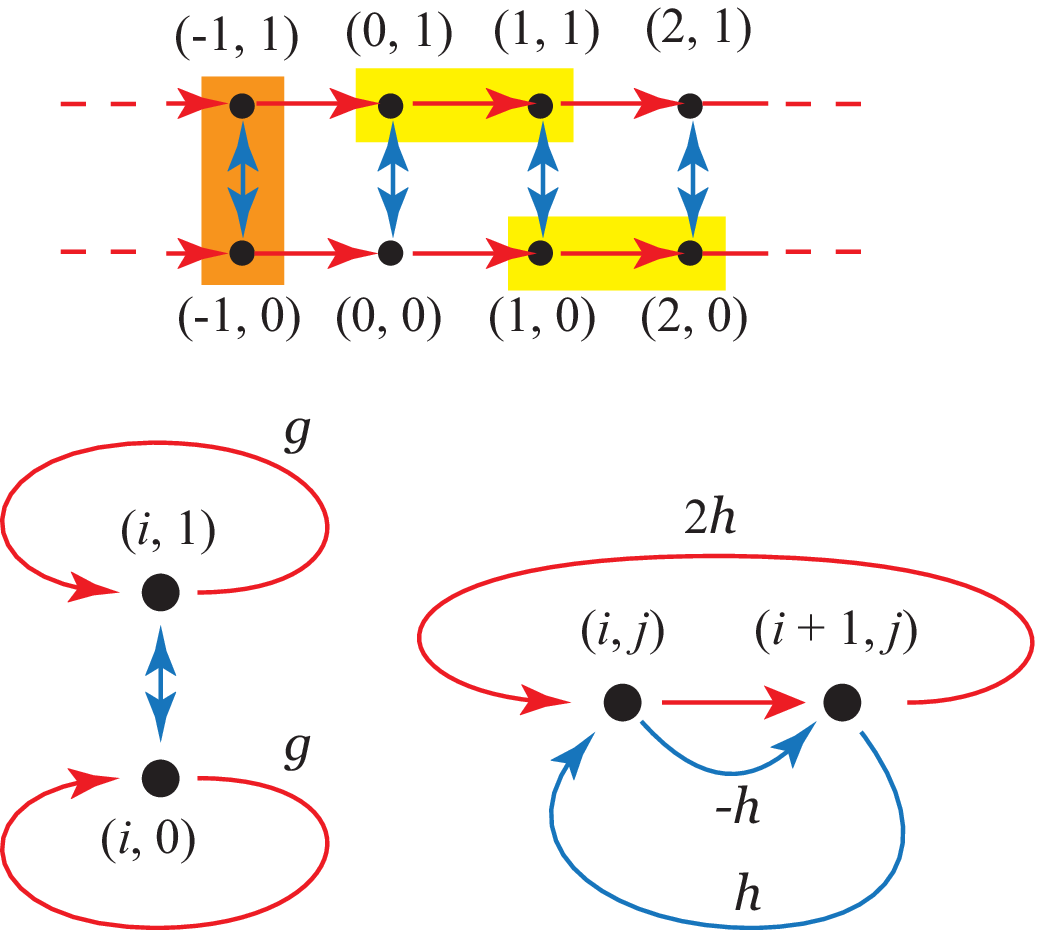} 
\put(0,65){$\Gamma$}
\put(-7,20){$\Delta_1$}
\put(100,20){$\Delta_2$}
\end{overpic}\qquad
\begin{overpic} 
[width=0.4\textwidth,trim=0 {-4cm} 0 0]{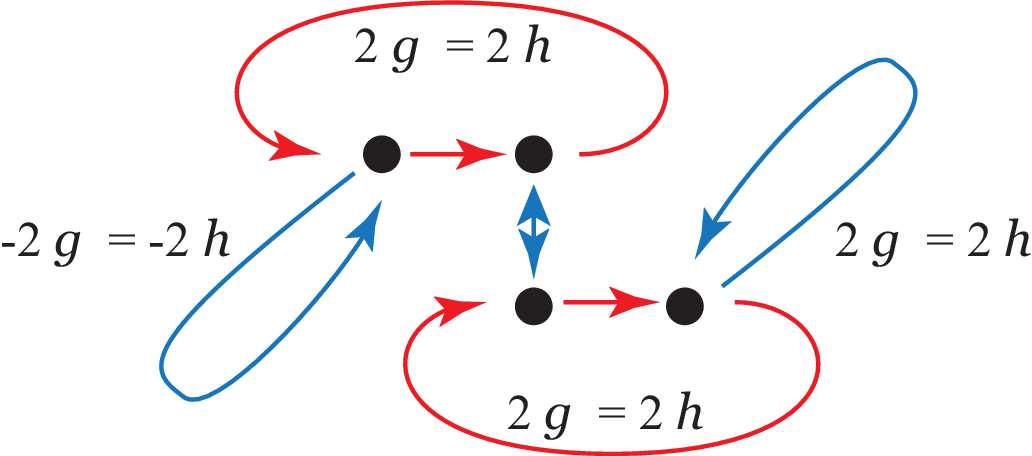} 
\put(40,75) {$\Delta_\top$}
\end{overpic}}
\caption{Consider the two quotient      graphs of $\Gamma$ at lower left and right: $\Delta_1=\Gamma / \langle g \rangle$ and 
$\Delta_2=\Gamma / \langle h \rangle$ from Figure~\ref{fig:upwards}(b). The  voltages on the quotient graphs are shown (if no voltage is given, the voltage is $0$).
One lift of $\Gamma / \langle g \rangle$ in $\Gamma$ has an   orange background while two lifts of 
  $\Gamma / \langle h \rangle$ have yellow backgrounds. The two graphs $\Delta_1$ and $\Delta_2$  are homomorphic images of the graph $\Delta_\top\cong \Gamma/\langle 2h\rangle$.}	  
\label{not_free}
\end{figure}
For  the voltage assignment to 
 $\Gamma/\langle g \rangle$, there are two possible choices of spanning 
 tree:  the tree consisting of the upward blue edge, and the tree 
 consisting of the downwards blue edge.
Meanwhile, for $\Gamma/\langle h \rangle$, each of the four edges would 
 serve as a spanning tree set.
For example, the rightwards red edge from the vertex denoted $(i, j)$ to 
 the vertex denoted $(i + 1, j)$ is a spanning tree set, which is lifted 
 to the set of tree sets
  \[
 \{ ((2i - 1, j), (2i, j)) : i \in \bbbZ \; \& \; j \in \{ 0, 1 \} \}
  \]
 as shown in Figure~\ref{not_free}.

}\end{exam}

Example~\ref{notfree} suggests that if the underlying graphs $\Delta_1$ and $\Delta_2$ of the two 
 voltage graphs are different, in order to answer the question whether $\Gamma_1\cong \Gamma_2$, we should seek a common cover $\Delta_\top$ of $\Delta_1$ and $\Delta_2$,  lift the voltages from these two voltage graphs 
 to the common cover (constructed in Prop.~\ref{Lift_voltages-new}), and compare the two lifted voltage assignments (i.e., take $\Delta_\top$ to play the role of $\Delta$ in Theorem~\ref{thm:iso-voltage0}).
This brings us to the main question we address here. 
Suppose that two finite connected 
 graphs $\Delta_1$ and $\Delta_2$ are given with respective condensed voltage assignments $\gamma_i:E_{\Delta_i}\to G$ ($i=1,2$)
 for some group $G$ and let $G_i=\langle\gamma_i(E_{\Delta_i})\rangle$.
 If $\Gamma_i$ ($i=1,2)$ are the respective derived graphs $\Gamma_i=\Delta\times_{\gamma_i} G_i$, necessarily connected by Proposition~\ref{prop:condensed_voltages}, what properties provide  $\Gamma_1\cong \Gamma_2$?
For this purpose, as observed, we  pull together Theorem~\ref{thm:iso-voltage0} and 
 Proposition~\ref{Lift_voltages-new} to obtain a condition that determines isomorphism of connected `highly symmetric' graphs, $\Gamma_1$ and $\Gamma_2$, finite or infinite. 
 \begin{figure}[ht]
\begin{center}
 \begin{tikzcd}
       \Gamma_1 = \Delta_1\times_{\gamma_1} G_1 
        \arrow[dd, "{\phi_1}" '] 
        \arrow[rd,"{\nu_1}",dashrightarrow]
        \arrow[rr, "{\varphi}", "{\cong}" '] & &
        \Gamma_2=\Delta_2\times_{\gamma_2} G_2
        \arrow[dd,  "{\phi_2}"] 
        \arrow[ld,"{\nu_2}" ',dashrightarrow]\\
        & \Delta_\top
        \arrow[ld, "{\mu_1}"  ] 
        \arrow[rd, "{\mu_2}" ' ] 
        & \\
        (\Delta_{1},\gamma_1) &  & (\Delta_2,\gamma_2)
    \end{tikzcd}
\end{center}
\caption{Situation for $\Delta_1\not \cong \Delta_2$ where, for $i=1,2$, $\Delta_i$ is a connected finite graph and $\gamma_i$ is a condensed voltage assignment in $\Delta_i$. Here $G_i=\langle\gamma_i(E_{\Delta_i})\rangle$.}\label{Dcommutes}
\end{figure}

 We note that if there is a joint cover $\Delta_\top$ for two graphs $\Delta_1$ and $\Delta_2$ then neighborhood of vertices in $\Delta_1$ must `match' neighborhoods of vertices in $\Delta_2$. We exploit this within the directed product $\Theta=\Delta_1\times\Delta_2$. 
 By Corollary~\ref{cor:cover-in-product}, if a common cover $\Delta_\top$ from Proposition~\ref{prop:upwards} exists, it is isomorphic to a subgraph of $\Theta$ and the coordinate projection maps $(v_1,v_2)\mapsto v_i$  and $(e_1,e_2)\mapsto e_i$ for $i=1,2$ act as coverings $\mu_i$ in Figure~\ref{Dcommutes}.
 Therefore it follows that a common cover of $\Delta_1$ and $\Delta_2$, if it exists, is a subgraph of $\Theta$, while the cover maps $\mu_i$ are the coordinate projections.

Given that a connected subgraph $\Delta_\top$ of $\Theta$  is a common cover for $\Delta_1$ and $\Delta_2$,  by Proposition~\ref{prop:good-cover} there are coverings $\nu_i:\Gamma_i\to \Delta_\top$ such that $\phi_i=\mu_i\nu_i$ if and only if $\Delta_\top$ is a good common cover for both $\Delta_1$ and $\Delta_2$.

We observe that $\Theta$ 
may contain more than one good common cover, as shown in  Figure~\ref{fig:counterexample_query}.

In order to apply Proposition~\ref{Lift_voltages-new}, we need both $\mu_1$ and $\mu_2$ to be regular. 
If $G_1$ and $G_2$ are Dedekind, then all their subgroups are normal, and by Theorem~\ref{thm:sunada}, $G_{\mu_i}=G_i/G_{\nu_i}$ and $\mu_i$ are regular ($i=1,2$). 
Otherwise, one needs to determine whether $G_{\nu_i}$ is normal in $G_i$, $i= 1, 2$. 

\begin{rema}
For $i = 1, 2$, $G_{\nu_i}$ is normal in $G_i$ if and only if, for each 
 $\gamma_i(e)$, $e \in E_{\Delta_i}$ and each $\gamma_i'(e')$, $e' \in 
 E_{\top}$, $\gamma(e) \gamma'(e') \gamma(e)^{-1} \in G_{\nu_i}$.
\end{rema}

A good common cover $\Delta_\top$ is said to be {\em unresolved} if at least one of $\mu_1$ or $\mu_2$ are not regular.

Given a good common cover $\Delta_\top$ with regular $\mu_i$ ($i=1,2$), $\Gamma_i\cong \Delta_\top \ttimes {\gamma_i'} G_{\nu_i}$ by Proposition~\ref{Lift_voltages-new}, and it remains to determine if the lifted and condensed 
 voltage assignments  $\gamma_i'$ on $\Delta_\top$ give  $\Delta_\top \ttimes 
 {\gamma_1'} G_{\nu_1} \cong \Delta_\top \ttimes {\gamma_2'} G_{\nu_2}$ which is determined by Theorem~\ref{thm:iso-voltage0}.

\begin{figure}[th]
\begin{center}
\includegraphics*[width = 2.2in]{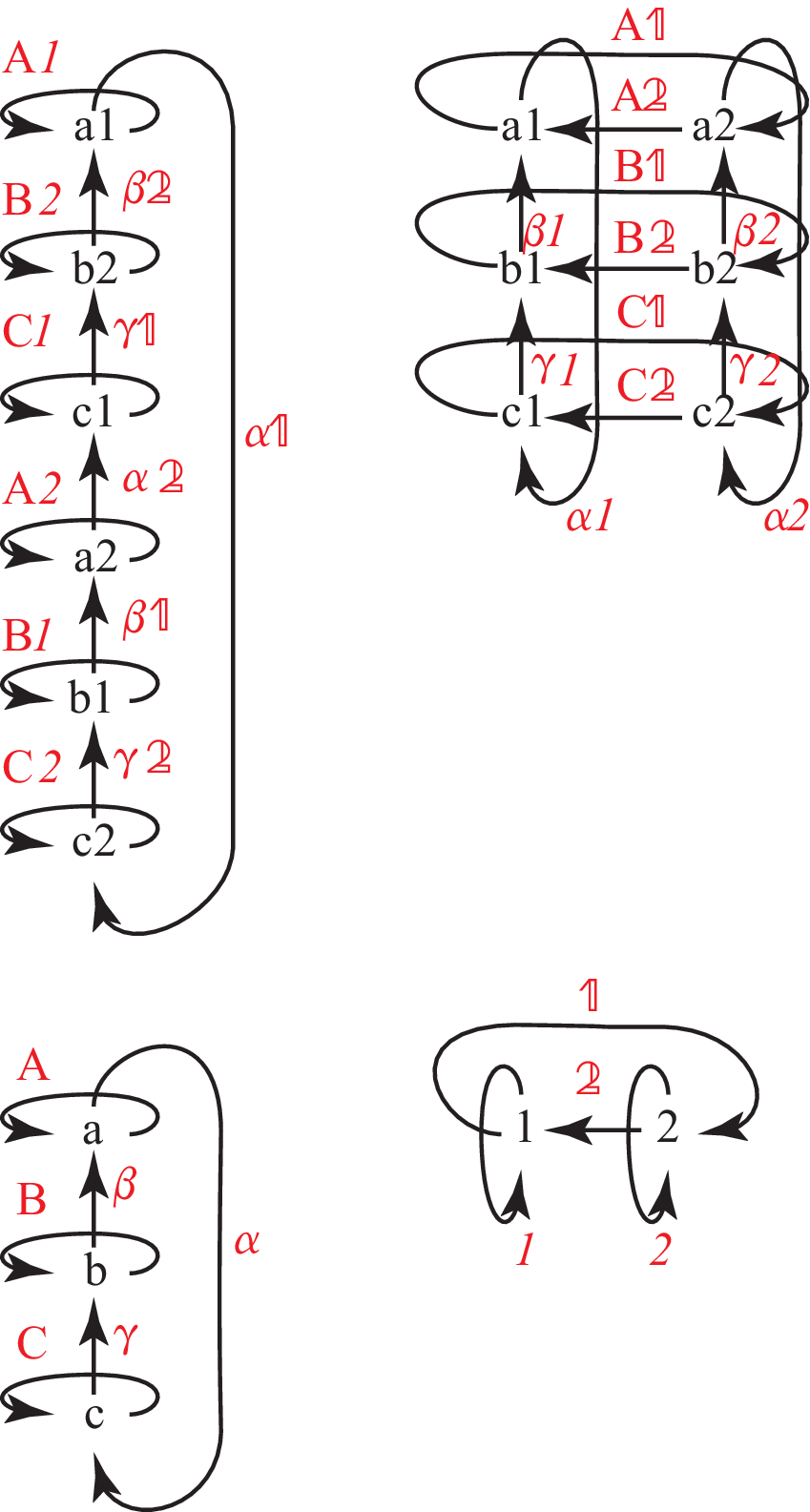} \end{center}
\caption{The vertices are labeled in black, the edges in red.
         The top two graphs are good common covers of the bottom two, and 
          the covers preserve edge and vertex labels.}
\label{fig:counterexample_query}
\end{figure}

Notice that every automorphism of $\Delta_{\top}$ that fixes a spanning tree $S_\top$ permutes edges that 
 share initial and terminal endpoints.
For each such automorphism $\alpha$ of $\Delta_\top$,
 let $F_{\alpha} : \gamma_1'(E_\top) \to \gamma_2'(E_\top)$ be defined by 
 $\gamma'_1(e) \mapsto \gamma'_2(\alpha(e))$ for $e\in E_\top$.
Because $\mu_i$ ($i=1,2$) are regular, by 
 Proposition~\ref{Lift_voltages-new}, 
 $\Delta_\top \times_{\gamma_1'} 
 \langle \gamma_1'(E_\top) \rangle \cong \Delta_1 \times_{\gamma_1}  G_1=\Gamma_1$, 
 and similarly for $\Delta_2$.

\begin{defn}
Let $\Delta_\top$
be a good common cover of $\Delta_1$ and $\Delta_2$ with regular coverings $\mu_1$ and $\mu_2$. We say that $\Delta_\top$ is {\em successful} if 
there exists an automorphism $\alpha$ on $\Delta_\top$ fixing a spanning tree $S_\top$ such that $F_\alpha$ extends to 
 a group isomorphism $\langle \gamma_1'(E_\top) \rangle \to \langle \gamma_2'(\alpha E_\top) 
 \rangle$, where $\gamma_i'$ is a condensed lift of $\gamma_i$ with respect to $S_\top$.

If there is no such $\alpha$ we say that $\Delta_\top$ {\em fails}.
\end{defn}

Thus, if $\Delta_\top$ is successful, then by Theorem~\ref{thm:iso-voltage0} we conclude that $\Delta_\top 
 \times_{\gamma_1'} G_{\nu_1} \cong  \Delta_\top \times_{\gamma_2'} G_{\nu_2}$, 
 and hence $\Gamma_1=\Delta_1 \times_{\gamma_1}  G_1 \cong \Delta_2 \times_{\gamma_2} 
 G_2=\Gamma_2$. On the other hand, if $\Delta_\top$ fails, $\Delta_1 \ttimes {\gamma_1} G_1 \not\cong 
 \Delta_2 \ttimes {\gamma_2} G_2$. 

Summing all above we have:

\begin{prop}
If $\Theta$ has as a subgraph  a successful good common cover $\Delta_\top$ for $\Delta_1$ and $\Delta_2$,  then $\Delta_1 \ttimes 
 {\gamma_1} G_1 \cong \Delta_2 \ttimes {\gamma_2} G_2$.
If there is a good common cover that fails, or if there 
 is no good common cover, then $\Delta_1 \ttimes {\gamma_1} G_1 \not\cong 
 \Delta_2 \ttimes {\gamma_2} G_2$.
\end{prop}

If there exist a good common cover as a subset of $\Theta$, but all such good covers are unresolved, we cannot conclude whether $\Gamma_1$ is isomorphic to $\Gamma_2$.

In particular, when $G_1, G_2 \le G$ where $G$ is abelian, the covers $\mu_i$ ($i=1,2$) are regular, hence we have the following theorem. 

\begin{theo}
Let $G$ be an abelian group. Let   $(\Delta_i,\gamma_i)$ $(i=1,2)$ be  a finite connected graph with  a condensed voltage assignment $\gamma_i$ 
such that $G_i=\langle \gamma_i(E_{\Delta_i})\rangle \le G$. Then $\Delta_1 \ttimes {\gamma_1} G_{1} \cong 
 \Delta_2 \ttimes {\gamma_2} G_{2}$ if and only if $\Delta_1\times\Delta_2$ has a connected subgraph that is a successful good common cover of 
 $\Delta_1$ and $\Delta_2$.
\end{theo}

Given a finitely generated abelian group $G$ we assume that $G$ is given with a standard presentation as is implied from the Fundamental Theorem of Finitely Generated Abelian Groups as a direct products of free abelian group with direct products of finite cyclic groups. 

\begin{cor}
Let $G$ be a finitely generated abelian group. Let   $(\Delta_i,\gamma_i)$, $i=1 ,2$, be  finite connected graphs with condensed voltage assignments $\gamma_i$ 
such that $G_i=\langle \gamma_i(E_{\Delta_i})\rangle \le G$. It is decidable whether $\Delta_1 \ttimes {\gamma_1} G_{1} \cong 
 \Delta_2 \ttimes {\gamma_2} G_{2}$.
\end{cor}

\begin{proof}
    Suppose $(\Delta_i,\gamma_i)$, $i=1, 2$, be finite connected graphs with condensed voltage assignments $\gamma_i$ such that $G_i=\langle \gamma_i(E_{\Delta_i})\rangle \le G$ where $G$ is abelian. Note that $\Delta_i \ttimes {\gamma_i} G_{i}$ may be infinite, but by Proposition~\ref{prop:condensed_voltages}, both are connected. 
We observe that vertex figures in $\Delta_1$ must match those of $\Delta_2$, and this can be checked by listing degrees of vertices in an increasing order. If this fails, the isomorphism does not exist. We assume that $G$ is finitely presented as indicated by the Fundamental Theorem of Finitely Presented Abelian Groups.  We also assume that voltage assignments $\gamma_i$ are using the same presentation. The graph $\Theta=\Delta_1\times \Delta_2$ is constructible. 
The number of connected subgraphs of $\Theta$ that cover both $\Delta_1$ and $\Delta_2$
are finite, and one can check whether any of these subgraphs are common covers by checking the projection maps $\mu_i$ (Corollary~\ref{cor:cover-in-product}).
If such common cover is not found,  the isomorphism does not exist (Proposition~\ref{prop:upwards}). If $\Delta_\top$ is a common cover, since each $\Delta_i$ is finite, and $G$ is abelian, hence has a solvable word problem, one can check whether $\Delta_\top$ is a good common cover. That is, whether every cyclic path with voltage being identity in $\Delta_i$ lifts to cyclic paths in $\Delta_\top$. 
If $\Delta_\top$ is not a good common cover, the isomorphism does not exist (by Proposition~\ref{prop:good-cover}, covers $\nu_i$ do not exist). Lifting voltages $\gamma_i$ in $\Delta_\top$ is constructable, since $\mu_i$ are projection maps. 
We choose a spanning tree set $S_\top$ in $\Delta_\top$ and condense the lifts of $\gamma_i$ with respect to $S_\top$ to obtain $\gamma_i'$. Since $\Delta_\top$ is finite and $G$ has a solvable word problem, the voltage assignments $\gamma_i'$ as described in Definition~\ref{def:condensation} are constructible. 
Finally, for any automorphism $\alpha:\Delta_\top\to\Delta_\top$ that fixes $S_\top$, i.e., permutes parallel edges,  we check whether the map $\gamma_1'(e)\mapsto\gamma_2'(\alpha(e))$ extends to an isomorphism 
(since each $\Delta_i \ttimes {\gamma_i} G_{i}$ is connected, Theorem~\ref{thm:iso-voltage0} applies), which for subgroups of finitely presentable abelian groups is decidable. 
\end{proof}


\section{Directions for Future Research}


We showed a constructive method to determine whether the derived graphs of two voltage graphs  $(\Delta_1, \gamma_1)$ and $(\Delta_2, \gamma_2)$ are isomorphic, when $\gamma_1$ and $\gamma_2$ are voltages from a commutative or Hamiltonian group. The method requires construction of a 
 good common cover $\Delta_{\top}$ with associated coverings $\mu_i$ and 
 $\nu_i$. In general, however, even when $\Delta_\top$ is a good common cover, the projection maps $\Delta_\top \to \Delta_i$ may not be regular, and hence, Proposition~\ref{Lift_voltages-new} does not apply. This is not a spurious issue, for there are examples of this situation. 


For example, choose a finite group $G$ with two subgroups $G_1$ and $G_2$ 
 such that $G_1 \cap G_2$ is not normal in at least one of $G_1$ or $G_2$.
One possible construction: choose a group $G_1'$ with an 
 subgroup $H < G_1'$ such that $H$ is not normal, and for any finite group $K$, let $G = G_1' 
  \times K$.
 Let $G_1 = G_1' \times \{ 1_K \}$ and $G_2 = H \times K$.
 Then $G_1 \cap G_2 = H \times \{ 1_K \}$ is not normal in $G_1$.
By Frucht's theorem (\cite{Frucht}), there is a graph $\Gamma$ such that 
 the automorphism group of $\Gamma$ is isomorphic to $G$.
In fact, this automorphism group acts freely on $\Gamma$ as it is the 
 result of taking a colored Cayley graph of $G$ (using generators 
 for $G_1 \cap G_2$, and then generators for $G_1$ and then generators for 
 $G_2$) and replacing the colors with ``gadgets'' representing colors
 (see, e.g., \cite[Theorem 3.3]{LS}).
Then $G$ acts freely on $\Gamma$, and following the constructions of Section~\ref{LVI}, 
 we obtain $\Delta_{\top} \cong \Gamma / (G_1 \cap G_2)$, so that for each $i$, the 
 composition $\Gamma \stackrel{\nu_i}{\to} \Delta_{\top} 
 \stackrel{\mu_1}{\to} \Delta_i$ has $G_{\nu_i} = G_1 \cap G_2 = H \times \{ 1_K \}$.
But $H \times \{ 1_K \}$ is not normal in $G_1$, then $\mu_1$ is not 
 regular, so the construction in Section~\ref{LVI} produces no answer.

The general construction that we described in Section~\ref{LVI} is abstract, and in case of arbitrary groups may not be algorithmic. For example, getting a good cover in $\Delta_1\times\Delta_2$ (Proposition~\ref{prop:good-cover}) may need a solvable word problem, or, applying Theorem~\ref{thm:iso-voltage0} requires solvable group isomorphism. However, choosing a good presentation of an abelian group, the method, although brute force, provides an algorithm that decides whether two voltage graphs have isomorphic derived graphs. 
It would be of interest to 
 develop a more tractable algorithm that could be implemented for specific voltage graphs.
This would be useful as graph isomorphism is used to classify graphical 
 representations of the atomic or molecular structure of covalent 
 crystals (\cite{RCSR}).
We would not expect a PTIME algorithm, and the question of determining precise computational 
 complexity for specific classes of voltage graphs may be of interest within the complexity theory, as well as in graph theory.

\bibliographystyle{plain}
\bibliography{JKM}

\end{document}